\documentclass[a4paper,12pt, reqno]{amsart}
\usepackage{amsmath,amssymb,enumerate,epsfig,xcolor}

\usepackage{graphicx}
\usepackage{import}
\usepackage{xifthen}
\usepackage{pdfpages}
\usepackage{transparent}

\usepackage{comment}

\usepackage{float}
\usepackage{tikz-cd}

\theoremstyle{plain}
\newtheorem{thm}{Theorem}
\newtheorem*{thm*}{Theorem}
\newtheorem{cor}[thm]{Corollary}
\newtheorem*{cor*}{Corollary}
\newtheorem{prop}[thm]{Proposition}
\newtheorem{lemma}[thm]{Lemma}
\newtheorem{q}[thm]{Question}

\newtheorem*{conj*}{Conjecture}

\theoremstyle{definition}
\newtheorem{defi}[thm]{Definition}
\newtheorem*{defi*}{Definition}

\newtheorem*{exple*}{Example}

\newtheorem*{exples*}{Examples}

\theoremstyle{remark}
\newtheorem*{Rmk}{Remark}
\newtheorem*{Rmks}{Remarks}
\newtheorem{rmk}[thm]{Remark}
\newtheorem{rmks}[thm]{Remarks}
\newtheorem{claim}{Claim}

\newenvironment{pf}[1][Proof.]{\noindent \textbf{#1:} }{}
\newenvironment{enui}{\begin{enumerate}[(i)]}{\end{enumerate}}
\newenvironment{enua}{\begin{enumerate}[(a)]}{\end{enumerate}}

\newcommand\reff[1]{(\ref{#1})}

\newcommand\wt[1]{{\widetilde{#1}}}
\newcommand{\N}{\mathbb{N}}
\newcommand{\Z}{\mathbb{Z}}
\newcommand{\R}{\mathbb{R}}
\newcommand\Om{\Omega}
\newcommand\om{\omega}
\newcommand\x{\times}

\newcommand\al{\alpha}

\renewcommand\phi{\varphi}

\newcommand\lam{\lambda}
\newcommand\eps{\varepsilon}
\newcommand{\BAR}[1]{{\overline{#1}}}
\newcommand\wo{\setminus}

\newcommand\id{{\operatorname{id}}}
\newcommand\nn{{\nonumber}}
\newcommand\sub{\subseteq}

\newcommand\omst{\om^{\operatorname{st}}}

\newcommand{\Int}{\operatorname{Int}}

\newcommand\HH{\mathcal{H}}
\newcommand\si{\sigma}
\newcommand\const{\equiv}
\newcommand\vv{\mathbf{v}}
\newcommand\follows{\Leftarrow}
\newcommand\ddisp{D_{\operatorname{disp}}}
\newcommand\Ddisp{D_{\operatorname{disp}}^{\operatorname{rect}}}
\newcommand\dsq{D_{\operatorname{sq}}}
\newcommand\Dsq{D_{\operatorname{sq}}^{\operatorname{rect}}}
\newcommand\T{\mathbb{T}}
\newcommand\diam{\operatorname{diam}}
\newcommand\SSS{{\mathcal{S}}}
\newcommand\Wlog{w.l.o.g.~}
\newcommand\de{\delta}
\newcommand\then{\Longrightarrow}
\newcommand\U{\mathcal{U}}
\newcommand\V{\mathcal{V}}

\newcommand\F{\mathcal{F}}
\newcommand\im{{\operatorname{im}}}
\newcommand\dom{{\operatorname{dom}}}

\newcommand\dE{d_{\operatorname{E}}}
\newcommand\inj{\hookrightarrow}
\newcommand\Ham{{\operatorname{Ham}}}

\newcommand\Embs{{\operatorname{Emb}_{\operatorname{symp}}}}
\newcommand\omcan{\om^{\operatorname{can}}}
\newcommand\iso{\cong}

\begin{document}

\title[Instantaneous Hamiltonian displaceability, squeezability]{Instantaneous Hamiltonian displaceability and arbitrary symplectic squeezability for critically negligible sets}
\author{Yann Guggisberg\and Fabian Ziltener}
\address{Affiliation of Y.~Guggisberg: Utrecht University\\
mathematics institute\\
Hans Freudenthalgebouw\\
Budapestlaan 6\\
3584 CD Utrecht\\
The Netherlands
}
\email{y.b.guggisberg@uu.nl}
\address{Affiliation of F.~Ziltener: ETH Z\"{u}rich\\
departement of mathematics\\
R\"{a}mistrasse 101\\
8092 Z\"{urich}\\
Switzerland
}
\email{fabian.ziltener@math.ethz.ch}
\thanks{Y.~Guggisberg's work on this publication is part of the project \emph{Symplectic capacities, recognition, discontinuity, and helicity} (project number 613.009.140) of the research programme \emph{Mathematics Clusters}, which is financed by the Dutch Research Council (NWO). We gratefully acknowledge this funding.}
	
\setcounter{tocdepth}{1}
	
\maketitle

\begin{abstract} We call a metric space $s$-negligible iff its $s$-dimensional Hausdorff measure vanishes. We show that every countably $m$-rectifiable subset of $\R^{2n}$ can be displaced from every $(2n-m)$-negligible subset by a Hamiltonian diffeomorphism that is arbitrarily $C^\infty$-close to the identity. As a consequence, every countably $n$-rectifiable and $n$-negligible subset of $\R^{2n}$ is arbitrarily symplectically squeezable. Both results are sharp w.r.t.~the parameter $s$ in the $s$-negligibility assumption. 

The proof of our squeezing result uses folding. Potentially, our folding method can be modified to show that the Gromov width of $B^{2n}_1\wo A$ equals $\pi$ for every countably $(n-1)$-rectifiable closed subset $A$ of the open unit ball $B^{2n}_1$. This means that $A$ is not a barrier.
\end{abstract}

\tableofcontents

\section{Main results and a potential application regarding nonbarriers}\label{sec:main}
This article is concerned with the following question. We denote $\N_0:=\{0,1,\ldots\}$. Let $n\in\N_0$ and $(M,\om)$ be a symplectic manifold of dimension $2n$. 
\begin{q}\label{q:small} What is the smallest size of a subset of $M$ that carries some symplectic geometry?
\end{q}
In this article we interpret \emph{size} as \emph{Hausdorff dimension}. To recall this notion, let $(X,d)$ be a metric space and $s\in[0,\infty)$. For every $A\sub X$ we denote by $\diam A$ the diameter of $A$ w.r.t.~$d$.
\begin{defi}[Hausdorff negligibility, Hausdorff dimension]\label{def:neg Hd dim} We call $(X,d)$ \emph{$s$-\mbox{(Hausdorff-)} negligible} iff for every $\eps\in(0,\infty)$ there exists a countable collection $\SSS$ of subsets of $X$ that covers $X$, such that
\[\sum_{A\in\SSS}(\diam A)^s<\eps.%
\footnote{Here we use the convention $0^0:=1$.} 
\]
We define
\begin{align*}&\textrm{Hausdorff dimension of }(X,d)\\
&:=\inf\big\{s\in[0,\infty)\,\big|\,(X,d)\textrm{ is $s$-negligible}\big\}.\end{align*}
\end{defi}
\begin{Rmk}$(X,d)$ is $s$-negligible iff its $s$-dimensional outer Hausdorff measure vanishes.
\end{Rmk}
\subsection*{Instantaneous Hamiltonian displaceability of small sets}
One instance of Question \ref{q:small} is the following.
\begin{q}\label{q:displ} What is the smallest possible Hausdorff dimension $\ddisp$ of a subset of $M$ with positive (Hamiltonian) displacement energy?
\end{q}
It is well-known that $\ddisp\leq n$, if $(M,\om)$ is geometrically bounded. This follows from the fact that the displacement energy of every nonempty closed%
\footnote{i.e., compact without boundary} %
Lagrangian submanifold of such a symplectic manifold is positive, see e.g.~\cite{Che}. Apart from this estimate and a result%
\footnote{Theorem \ref{thm:inf displ} below} %
concerning \emph{submanifolds} of dimension less than $n$, to our knowledge, Question \ref{q:displ} has been completely open so far.

We now consider $M=\R^{2n}$ and equip this manifold with the standard symplectic form $\omst$. Our first main result implies that in this case $\Ddisp\geq n$ and therefore $\Ddisp=n$, where $\Ddisp$ is the variant of $\ddisp$ in which we only allow for subsets of $\R^{2n}$ that are bounded and countably $n$-rectifiable. This answers a variant of Question \ref{q:displ}.

In order to state our result, we need the following. Let $(X,d)$ be a metric space. Let $m\in\N_0$.
\begin{defi}[countable $m$-rectifiability]\label{defi:count m rect} We call $X$ \emph{countably $m$-rectifiable} iff there exists a countable%
\footnote{This includes empty and finite.} 
set $\F$ of Lipschitz maps with domains given by subsets of $\R^m$ and target $X$, such that
\begin{equation}\label{eqn_covering_A}
X=\bigcup_{f\in\F}\im(f).
\end{equation}
\end{defi}
The next lemma characterizes countable $m$-rectifiability.
\begin{lemma}[countable $m$-rectifiability]\label{le:count rect}The following conditions are equivalent:
\begin{enua}
\item\label{le:count rect:def}$(X,d)$ is countably $m$-rectifiable.
\item\label{le:count rect:bdd} There exists a set $\F$ as in Definition \ref{defi:count m rect}, such that the domain of each map in $\F$ is bounded.
\item\label{le:count rect:loc Lip} There exists a surjective \emph{locally} Lipschitz map from some subset of $\R^m$ to $X$.
\end{enua}
\end{lemma}
For a proof of this lemma see p.~\pageref{proof:le:count rect} in the appendix.
\begin{Rmk}[term \emph{rectifiability}] In \cite[3.2.14, p.251]{Fed69} H.~Federer calls $(X,d)$ \emph{countably $m$-rectifiable} iff it satisfies condition \reff{le:count rect:bdd}. By Lemma \ref{le:count rect} Federer's and our terminology coincide. It differs from the terminology used in some other sources.
\end{Rmk}
Our first main result is the following. We equip $\R^\ell$ with the Euclidean distance function and every subset of $\R^\ell$ with the corresponding restriction of this function. Let $n\in\N_0$.
\begin{thm}[instantaneous Hamiltonian displaceability of small sets]\label{thm:displ} Let $m\in\{0,\ldots,2n\}$, $A\sub\R^{2n}$ be a countably $m$-rectifiable subset and $B\sub\R^{2n}$ a $(2n-m)$-negligible\footnote{as in Definition \ref{def:neg Hd dim}} subset. There exists a linear function $H:\R^{2n}\to\R$, such that for almost every%
\footnote{This means that the set of all such $t$ has full (1-dimensional) Lebesgue measure.} 
$t\in\R$, we have
\[\phi_H^t(A)\cap B=\emptyset.%
\footnote{Here $\phi_H^t$ denotes the Hamiltonian time $t$ flow of $H$ (w.r.t.~the standard symplectic form). Since $H$ is linear, this flow is globally defined for all times.}
\]
\end{thm}
Theorem \ref{thm:displ} implies that $A$ and $B$ as in the hypothesis of this result are $C^\infty$-instantaneously displaceable. To explain this notion, we denote 
\begin{align}\label{eq:HH M om}\HH(M,\om):=\big\{&H\in C^\infty\big([0,1]\x M,\R\big)\,\big|\\
\nn&\forall t\in[0,1]:\,\phi_H^t:M\to M\textrm{ is well-defined and surjective}\big\}.
\end{align}
By a \emph{Hamiltonian diffeomorphism} we mean a map $\phi_H^1$, where $H\in\HH(M,\om)$.%
\footnote{Every such map is indeed a diffeomorphism of $M$. To see this, note that the time $t$ flow of a time-dependent vector field on a manifold $M$ is always an injective smooth immersion on its domain of definition. Hence if it is everywhere well-defined and surjective then it is a diffeomorphism of $M$.} 
\begin{defi}[instantaneous displaceability]\label{def:inst disp} Let $k\in\N_0\cup\{\infty\}$. We call two subsets $A$ and $B$ of a symplectic manifold \emph{$C^k$-instantaneously (Hamiltonianly) displaceable}, iff every neighbourhood of the identity in the weak $C^k$-topology contains a Hamiltonian diffeomorphism that displaces $A$ from $B$. 
\end{defi}
As mentioned, $A$ and $B$ as in Theorem \ref{thm:displ} are $C^\infty$-instantaneously displaceable.

By considering the case $A=B$, Theorem \ref{thm:displ} immediately implies the following corollary.
\begin{cor}[instantaneous Hamiltonian displaceability of a critically negligible set]\label{cor:displ} Let $A\sub\R^{2n}$ be a countably $n$-rectifiable and $n$-negligible subset. There exists a linear function $H:\R^{2n}\to\R$, such that for almost every $t\in\R$, we have
\begin{equation}\label{eq:phi t H}\phi^t_H(A)\cap A=\emptyset.\end{equation}
\end{cor}
\begin{Rmk}[critical negligibility] We call a subset of a $2n$-dimensional symplectic manifold \emph{critically (Hausdorff-)negligible}, iff its $n$-dimensional outer Hausdorff measure%
\footnote{w.r.t.~some Riemannian metric} 
vanishes. By Corollary \ref{cor:displ}, every countably $n$-rectifiable critically negligible subset of $\R^{2n}$ is $C^\infty$-instantaneously displaceable from itself. For a justification of the term \emph{critical} see Remark \ref{rmk:negli sharp}\reff{rmk:negli sharp:crit} below.
\end{Rmk}
\begin{exple*}[subset satisfying the hypotheses of Corollary \ref{cor:displ}] Let $S\sub\R^n$ be a subset of Lebesgue measure 0, and $F:S\to\R^{2n}$ a locally Lipschitz map. The image $A:=F(S)$ is a countably $n$-rectifiable and $n$-negligible subset of $\R^{2n}$ and therefore satisfies the hypotheses of Corollary \ref{cor:displ}. Here we used Lemma \ref{le:count rect}.
\end{exple*}
\begin{Rmk}[countably $m$-rectifiable and $s$-Hausdorff dimensional subset]Let $\ell\in\N:=\{1,2,\ldots\}$ and $s\in(0,\ell]$. We denote by $\lceil\rceil:\R\to\Z$ the ceiling function and $m:=\lceil s\rceil$. There exists a countably $m$-rectifiable, and $s$-Hausdorff dimensional subset $A$ of $\R^\ell$ with vanishing $s$-dimensional Hausdorff measure. Such a set satisfies the hypothesis of Corollary \ref{cor:displ}, if $\ell=2n$ and $s\leq n$.%
\footnote{Here we use that every $m$-rectifiable metric space is $m'$-rectifiable for every $m'\geq m$.} 
This shows that this corollary is a statement not only about integer-dimensional subsets, but subsets of dimension given by an arbitrary real number between 0 and $n$.

To construct $A$ as above, we choose a countable subset $S$ of the interval $(0,s)$. For every $s'\in[0,\ell]$ there exists a subset of $\R^m$ of Hausdorff dimension $s'$, e.g.~a suitable Cantor dust. This follows from \cite[Theorem 9.3]{Fal}. For each $s'\in S$ we choose such a subset $A_{s'}$. We denote by $0$ the origin in $\R^{2n-m}$ and define
\[A:=\left(\bigcup_{s'\in S}A_{s'}\right)\x\{0\}.\]
This set has the desired properties.
\end{Rmk}

To formulate the next application of Theorem \ref{thm:displ}, let $(M,\om)$ be a symplectic manifold. Recall the definition \eqref{eq:HH M om}. We denote
\[\Vert\cdot\Vert:\HH(M,\om)\to[0,\infty],\quad\Vert H\Vert:=\int_0^1\big(\sup_MH^t-\inf_MH^t\big)dt\]
and define the \emph{displacement energy} of a subset $A\sub M$ to be
\[e(A):=e(A,M,\om):=\inf\big\{\Vert H\Vert\,\big|\,H\in\HH(M,\om):\phi_H^1(A)\cap A=\emptyset\big\}.%
\footnote{Alternatively, one may define a displacement energy by using only functions $H$ with compact support. However, it seems more natural to allow for all functions in $\HH(M,\om)$. For some remarks on this topic see \cite{SZHofer}.}
\]
\begin{cor}[minimal Hausdorff dimension of set with positive displacement energy]\label{cor:inf d} We have
\begin{align*}\Ddisp:=\inf\Big\{&D\in[0,\infty)\,\Big|\,\exists A\sub\R^{2n}:\textrm{ $D$-Hausdorff dimensional,}\\
&\textrm{bounded, countably $\lceil D\rceil$-rectifiable, }e(A)>0\Big\}=n.
\end{align*}
\end{cor}
For a proof of this corollary see p.~\pageref{proof:cor:inf d}.

Corollary \ref{cor:inf d} provides an answer to the variant of Question \ref{q:displ} in which we only consider bounded subsets that are rectifiable in a suitable sense.
\begin{rmks}[sharpness of negligibility condition, criticality]\label{rmk:negli sharp}\begin{enui}\item\label{rmk:negli sharp:displ} The negligibility condition in Theorem \ref{thm:displ} is sharp in the sense that it does not suffice to assume that $B$ is $s$-negligible for some $s\in\big(2n-m,\infty)$, instead of $s=2n-m$. It does not even suffice to assume that it is $s$-negligible for \emph{every} $s\in\big(2n-m,\infty)$.

To see this, let $A$ and $B$ be submanifolds of $\R^{2n}$ without boundary of dimensions $m$ and $(2n-m)$, that intersect transversely at some point. Then $A$ is countably $m$-rectifiable and $B$ is $s$-negligible for every $s\in\big(2n-m,\infty)$.

However, the set of all continuous self-maps of $\R^{2n}$ that do \emph{not} displace $A$ from $B$, is a neighbourhood of the identity w.r.t.~the compact open topology.%
\footnote{The proof of this is based on the fact that the identity map on the $(m-1)$-dimensional sphere is not homotopic to a constant map.} 
It follows that $A$ and $B$ are not $C^0$-instantaneously Hamiltonianly displaceable. Hence the conclusion of Theorem \ref{thm:displ} does not hold.
\item\label{rmk:negli sharp:cor}Similarly, the negligibility condition in Corollary \ref{cor:displ} is sharp in the sense that it does not suffice to assume that $A$ is $s$-negligible for \emph{every} $s\in\big(n,\infty)$. To see this, we choose $A$ to be the union of two $n$-dimensional submanifolds of $\R^{2n}$ without boundary that intersect transversely in some point. Then $A$ is countably $n$-rectifiable and $s$-negligible for every $s\in(n,\infty)$, but not $C^0$-instantaneously displaceable from itself. This follows from assertion \reff{rmk:negli sharp:displ}.

There are even \emph{submanifolds} with these properties. For example, we may take $A$ to be any nonempty closed Lagrangian submanifold of $\R^{2n}$. See Proposition \ref{prop:Lag not C1 inst displ} below.
\item\label{rmk:negli sharp:crit} By Corollary \ref{cor:displ} and assertion \reff{rmk:negli sharp:cor} the minimal Hausdorff dimension $D$ of a not $C^\infty$-instantaneously displaceable countably $n$-rectifiable subset of $\R^{2n}$ is $D=n$. In this sense the number $n$ is critical. This justifies the term \emph{critically negligible}.
\end{enui}
\end{rmks}

The idea of proof of Theorem \ref{thm:displ} is to consider the displacement vector map
\[\vv:\R^{2n}\times\R^{2n}\to\R^{2n},\qquad\vv(x,y):=y-x.\]
A vector $v\in\R^{2n}$ disjoins $A$ from $B$ iff $v$ does not lie in the image of $A\x B$ under $\vv$. This image has Lebesgue measure 0. This follows from a result of geometric measure theory about the Hausdorff measure of the product of metric spaces. It follows that translation by almost every vector $v$ disjoins $A$ from $B$. Such a translation is the Hamiltonian diffeomorphism induced by some linear function. The conclusion of Theorem \ref{thm:displ} now follows from Fubini's theorem and the Change of Variables Theorem.
\subsection*{Arbitrary symplectic squeezing for a critically negligible set}
We now consider another instance of Question \ref{q:small}. In order to formulate it, we need the following.
\begin{defi}[symplecticity of a map on a subset, arbitrary squeezability]\label{defi:sympl} Let $(M,\om)$ and $(M',\om')$ be symplectic manifolds and $S\sub M$ a subset.
\begin{enui}
\item A map from $S$ to $M'$ is called \emph{(ambiently) symplectic} iff it extends to a symplectic map from some open neighbourhood of $S$ to $M'$.
\item We say that $S$ \emph{injectively (ambiently) symplectically maps into $M'$} iff there is an injective symplectic map from $S$ to $M'$.
\item\label{defi:sympl:emb}We say that $S$ \emph{(ambiently) symplectically embeds into $M'$} iff there exists a symplectic embedding of some open neighbourhood of $S$ into $M'$. 
\item We call a subset of a $2n$-dimensional symplectic manifold \emph{arbitrarily (symplectically) squeezable} iff it injectively symplectically maps into every nonempty $2n$-dimensional symplectic manifold.
\end{enui}
\end{defi}
\begin{rmks}[injective symplectic mappability, arbitrary squeezability]\label{rmk:inj sympl}\phantom{.}
\vspace{-1\baselineskip}
\begin{enui}
\item In this definition we have added the word \emph{ambiently}, as the term \emph{symplectic map} already has a weaker meaning if $S$ is a symplectic submanifold of $M$. In this case every ambiently symplectic map from $S$ to $M'$ is symplectic, but not vice versa. When considering a general subset $S$ of $M$, we leave out \emph{ambiently}, as confusion seems unlikely in this situation.
\item\label{rmk:inj sympl:emb} If $\dim M=\dim M'$, and $S$ is \emph{compact} and injectively symplectically maps into $M'$, then it symplectically \emph{embeds} into $M'$. This follows from Lemma \ref{le:loc inj compact} in the appendix.
\item A subset of a $2n$-dimensional symplectic manifold is arbitrarily squeezable iff it injectively symplectically maps into every open neighbourhood of the origin in $\R^{2n}$. The implication ``$\follows$'' follows from Darboux's theorem.
\end{enui}
\end{rmks}
\begin{q}[small not arbitrarily squeezable set]\label{q:emb} What is the smallest Hausdorff dimension $\dsq$ of a not arbitrarily squeezable subset of $\R^{2n}$?
\end{q}
It follows from \cite[Lemma 9]{SZsqueeze} that $\dsq\geq2$.%
\footnote{This lemma is based on Moser isotopy.} 
Hence in the case $n=1$ we have $\dsq=2$. In the case $n\geq2$ we have $\dsq\leq n$. This follows from \cite[Theorem 1]{SZsqueeze}. Using neighbourhood theorems and Gromov's isosymplectic embedding theorem, it is easy to show that certain isotropic and symplectic submanifolds of $\R^{2n}$ are arbitrarily squeezable.%
\footnote{See Propositions \ref{prop:isotr} and \ref{prop:sympl sq} below.} %
Apart from this, to our knowledge, Question \ref{q:emb} has been completely open so far. Our second main result is the following.
\begin{thm}[arbitrary symplectic squeezing for a critically negligible set]\label{thm:squeeze} Let $n\in\N_0$. Every bounded, countably $n$-rectifiable, and $n$-negligible subset of $\R^{2n}$ is arbitrarily symplectically squeezable.
\end{thm}
This theorem implies that $\Dsq\geq n$ and therefore
\[\Dsq=n,\]
where $\Dsq$ is the variant of $\dsq$ in which we only allow for subsets of $\R^{2n}$ that are bounded and countably $n$-rectifiable. This answers a variant of Question \ref{q:emb}.

\begin{rmk}[sharpness of the negligibility condition]\label{rmk:negli}The negligibility condition of Theorem \ref{thm:squeeze} is sharp in the sense that it does not suffice to assume that the subset is $s$-negligible for some real number $s\in(n,\infty)$, instead of $n$-negligible. It does not even suffice to assume that it is $s$-negligible for \emph{every} $s\in(n,\infty)$. This follows from Theorem \ref{thm:nonsqueezable} below.
\end{rmk}

\label{idea:thm:squeeze} Theorem \ref{thm:squeeze} is an application of Theorem \ref{thm:displ}. It is based on a lemma stating that every countably $n$-rectifiable and $n$-negligible subset $A$ of the product of an open (two-dimensional) rectangle $Q$ and a compact subset $K$ of $\R^{2n-2}$ injectively symplectically maps into $R\x U$, where $R$ is a rectangle of slightly more than half the size of $Q$, and $U$ is an arbitrary neighbourhood of $K$.

This follows from a folding construction, which we now sketch. We choose a small $\de>0$ and denote by $V_\de$ the rectangle $(0,1)\x(-1,1)$ with some $\de$-thickened slit removed, see Figure \ref{fig:V de proj A}. By rescaling, and using some two-dimensional symplectic embedding, we may assume \Wlog that $A$ is contained in $V_\de\x K$.

We choose a symplectomorphism $\psi$ of $\R^2$ that pushes down the upper right part of the rectangle $(0,1)\x(-1,1)$ to its lower right part. See Figure \ref{fig:psi folded}. We denote by $A_\pm$ the upper and lower parts of $A$, see Figure \ref{fig:V de proj A}. By Theorem \ref{thm:displ} there exists a linear function $H:\R^{2n}\to\R$ that $C^\infty$-instantaneously Hamiltonianly displaces $(\psi\x\id)(A_+)$ and $A_-$ from each other. By suitably cutting off this function, we obtain a function $\wt H$.

There exists a small $t$, such that the Hamiltonian flow $\phi_{\wt H}^t$ displaces $(\psi\x\id)(A_+)$ from $A_-$. The map given by $\phi_{\wt H}^t\circ(\psi\x\id)$ on $A_+$ and by the identity on $A_-$ maps $A$ injectively and symplectically into a slightly enlarged copy of the rectangle $(0,1)\x(-1,1)$ with the upper right part removed. By composing this map with the product of a suitable two-dimensional symplectic embedding and the identity, we obtain an injective symplectic map from $A$ into $R\x U$, as desired. This finishes the sketch of the folding construction.

In the proof of Theorem \ref{thm:squeeze} we will apply this folding construction several times to each factor of the product $\R^{2n}=\R^2\x\cdots\x\R^2$. We may view this as ``crumbling up'' the given subset of $\R^{2n}$.

\begin{Rmk}[folding and perturbing]The map $\psi\x\id$ folds $A$ back. The image of $A_+$ under $\psi\x\id$ may intersect $A_-$, see Figure \ref{fig:psi folded}. The trick in the folding construction is to remove this intersection by perturbing the map $\psi\x\id$. This means that we compose this map with the Hamiltonian diffeomorphism $\phi_{\wt H}^t$, which is $C^\infty$-close to the identity.
\end{Rmk}
\begin{figure}
\centering

	\def\svgwidth{1\columnwidth}
	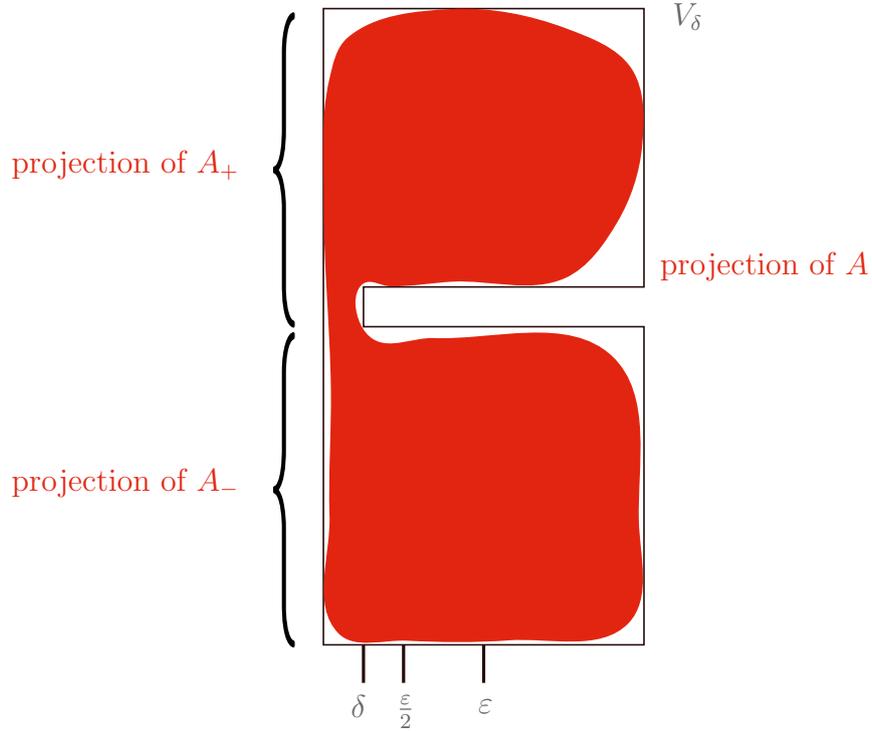

\caption{The set $V_\de$ and the projections of $A$ and $A_\pm$ to the first factor in $\R^{2n}=\R^2\x\cdots\x\R^2$.}
\label{fig:V de proj A}
\end{figure}

\begin{figure}
\centering

	\def\svgwidth{1\columnwidth}
	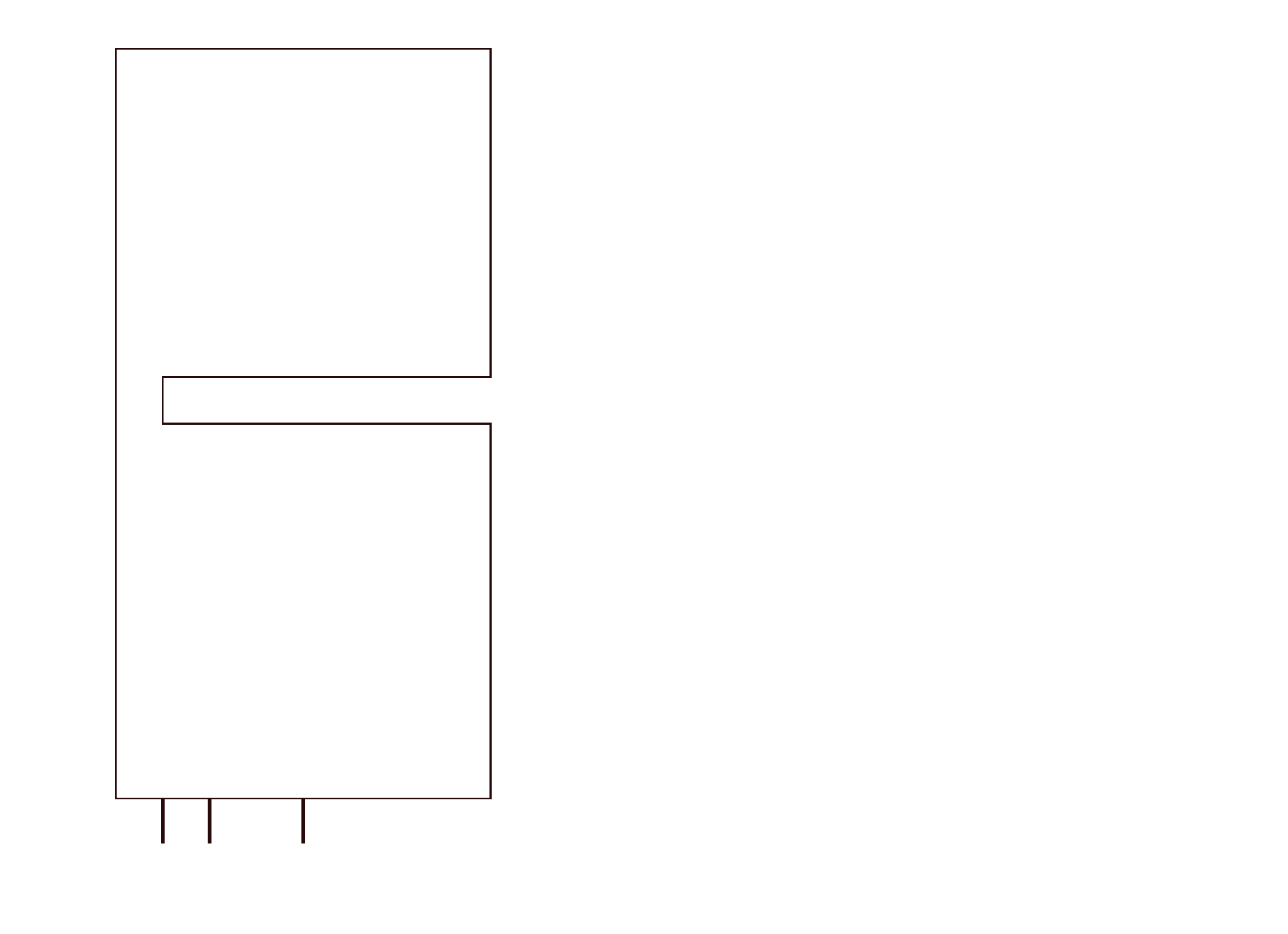

\caption{\emph{Left:} The map $\psi$. \emph{Right:} The map $\psi\x\id$ folds $A_+$ back, possibly making it intersect $A_-$.}
\label{fig:psi folded}
\end{figure}

\subsection*{Philosophical conclusion}
It is a philosophical principle that below half the dimension of a given symplectic manifold, symplectic geometry becomes flexible. Questions \ref{q:displ} and \ref{q:emb} put displacement and squeezing versions of this principle into the precise framework of subsets and Hausdorff dimension. Corollary \ref{cor:inf d} and Theorem \ref{thm:squeeze} confirm the principle in this framework.

Usually attention is restricted to \emph{submanifolds} of symplectic manifolds, which have \emph{integer} dimension. In contrast, in the setting of Corollary \ref{cor:inf d} and Theorem \ref{thm:squeeze} we consider arbitrary rectifiable \emph{subsets}, which may have any \emph{real} (Hausdorff) dimension. These results are sharp in terms of the real-valued Hausdorff dimension of the subset, and not just in terms of the integer dimension of some submanifold.
\subsection*{Potential application of our folding construction: Small sets are nonbarriers.}
We expect that the folding construction of the proof of Theorem \ref{thm:squeeze} can be modified to prove the following conjecture. We denote by $B^n_r$ the open ball in $\R^n$ of radius $r$ around $0$.
\newpage
\begin{conj*}Let $n\in\N_0$ and $A$ be a relatively closed subset of $B^{2n}_1$ that satisfies one of the following conditions:
\begin{enua}
\item $A$ is countably $(n-1)$-rectifiable.
\item $A$ is countably $n$-rectifiable and $(n-1)$-negligible.
\end{enua}
Then the Gromov width of $B^{2n}_1\wo A$ equals $\pi$.
\end{conj*}
Building on the terminology in P.~Biran's article \cite{Bir}, we define a \emph{nonbarrier} for a given symplectic manifold $(M,\om)$ to be a subset $A$ of $M$ such that $M\wo A$ and $M$ have the same Gromov width. The statement of the above conjecture implies that every subset $A$ as in the hypothesis of the conjecture is a nonbarrier. For $n=2$ a stronger version of the conjecture was proved by K.~Sackel, A.~Song, U.~Varolgunes, and J.~J.~Zhu in \cite[Theorem 1.4]{SSVZ}.

On the other hand, in \cite[Theorem 1.D]{Bir} P.~Biran showed that certain Lagrangian CW-complexes are \emph{Lagrangian} barriers. For recent work on Lagrangian barriers see \cite[Theorem 1.3]{SSVZ}, the work \cite[Theorem 1.5]{BS} of J.~Brendel and F.~Schlenk, and the work \cite{OS} of E.~Opshtein and F.~Schlenk. In \cite[Theorem 1.3]{HKHO} P.~Haim-Kislev, R.~Hind, and Y.~Ostrover showed that certain finite unions of symplectic subspaces of $\R^{2n}$ of codimension two are \emph{symplectic} barriers.

Our idea of proof for the conjecture is to modify the folding construction of the proof of Theorem \ref{thm:squeeze} (as outlined on p.~\pageref{idea:thm:squeeze}), by considering a Hamiltonian isotopy $\psi:[0,1]\x\R^2\to\R^2$ %
\footnote{instead of a single symplectomorphism of $\R^2$} 
that pushes down the upper right part of the rectangle $(0,1)\x(-1,1)$. We apply Theorem \ref{thm:displ} to obtain a Hamiltonian diffeomorphism that is $C^\infty$-close to $\id$ and displaces $\psi\x\id([0,1]\x A_+)$ from $A_-$. We then use a cut off argument to obtain a compactly supported Hamiltonian diffeomorphism that folds $A$ into a box of roughly half the original size.

By iterating this modified folding construction, we obtain a Hamiltonian diffeomorphism that maps $A$ into an arbitrarily small open set. We move this set close to the boundary of the ball $B^{2n}_1$ via some Hamiltonian diffeomorphism. The statement of the conjecture should now follow.

One challenge of this modified construction is to obtain a symplectomorphism of the \emph{ball} $B^{2n}_1$ that maps $A$ into some small open set, rather than a symplectomorphism of $\R^{2n}$. For this we need to apply a suitable Hamiltonian diffeomorphism that creates some space outside of $A$ and close to the boundary of $B^{2n}_1$. 
\subsection*{Organization of this article}
In Section \ref{sec:rel} we formulate some results that are related to Corollary \ref{cor:displ}, Theorem \ref{thm:squeeze}, and Question \ref{q:small} in general. In Section \ref{sec:proof main} we prove Theorems \ref{thm:displ} and \ref{thm:squeeze} and Corollary \ref{cor:inf d}, and in Section \ref{sec:proof prop} those results from Section \ref{sec:rel} that are not proved elsewhere. In the Appendix in Section \ref{sec:count rect loc inj} we prove Lemmas \ref{le:count rect} and \ref{le:loc inj compact}, which were used in Section \ref{sec:main}.
\subsection*{Acknowledgments}
We thank \'Alvaro del Pino G\'omez for mentioning an alternative approach for the proof of Theorem \ref{thm:inf displ} below and for suggesting to use the isosymplectic embedding theorem to prove Proposition \ref{prop:sympl sq} below. We also thank Felix Schlenk for some explanations regarding the isosymplectic embedding theorem and for feedback on a previous version of this article.

\section{Related results}\label{sec:rel}
\subsection*{Results related to Corollary \ref{cor:displ}}
Let $n\in\N_0$. The next result is a version of Corollary \ref{cor:displ} for a subset of $\R^{2n}$ given by a submanifold of subcritical dimension%
\footnote{By this we mean dimension strictly less than $n$.} 
. It has been known for a long time. It follows e.g.~from \cite[Theorem 1.1]{Gur}.%
\footnote{\cite[Theorem 1.1]{Gur} states the following: Let $N$ be a closed connected submanifold of a symplectic manifold, such that $N$ is nowhere coisotropic and its normal bundle admits a non-vanishing section. Then $N$ is infinitesimally Hamiltonianly displaceable. Theorem \ref{thm:inf displ} follows from this and the fact that a smooth vector bundle admits a nonvanishing section, if its rank is bigger than the dimension of its base. (See \cite[II.15.3, p.~115]{Bre}.)}
\begin{thm}[infinitesimal displaceability of submanifold of subcritical dimension]\label{thm:inf displ} Let $(M,\om)$ be a symplectic manifold of dimension $2n$ and $N\sub M$ a closed submanifold of dimension strictly less than $n$. Then $N$ is \emph{infinitesimally (Hamiltonianly)} displaceable, i.e., there exists a Hamiltonian vector field on $M$ that is nowhere tangent to $N$. 
\end{thm}
\begin{Rmk}[infinitesimal versus instantaneous displaceability and displacement energy]If a closed submanifold of $M$ is infinitesimally displaceable then it is $C^\infty$-instantaneously displaceable%
\footnote{See Definition \ref{def:inst disp}.} 
 and its displacement energy vanishes. This follows from an elementary argument involving the flow of an infinitesimally displacing Hamiltonian vector field, and the tubular neighbourhood theorem.
\end{Rmk}
For a submanifold $N$ of dimension $n$ the conclusion of Theorem \ref{thm:inf displ} remains true, provided that $N$ is closed, connected, and somewhere non-Lagrangian, and the normal bundle of $N$ admits a nonvanishing section. This was proved by F.~Laudenbach and J.-C.~Sikorav \cite[Theorem 1]{LS}. If $TN$ also has a Lagrangian complement then the same statement follows from L.~Polterovich's result \cite[Theorem 1.2]{Pol}.

\begin{Rmk} On the other hand, let $(M,\om)$ be a symplectic manifold and $L\sub M$ a nonempty closed Lagrangian submanifold. We claim that $L$ is not infinitesimally displaceable. To see this, let $H\in C^\infty(M,\R)$. The claim is a consequence of the following facts:
\begin{itemize}\item The restriction $f:=H|_L$ attains its extrema, which are critical points of $f$.
\item At each critical point of $f$ the Hamiltonian vector field $X_H$ is tangent to $L$.
\end{itemize}
\end{Rmk}

The next result states that certain Lagrangians are not \emph{instantaneously} displaceable, either.
\begin{prop}[instantaneous non-displaceability of Lagrangian]\label{prop:Lag not C1 inst displ} Every nonempty closed Lagrangian submanifold of $\big(\R^{2n},\omst\big)$ is not $C^1$-instantaneously displaceable. 
\end{prop}
For a proof of this proposition see Section \ref{sec:proof prop:Lag}. The proof is based on a result of M.~Gromov, which states that a Hamiltonian isotopy does not displace the zero-section of the cotangent bundle of a nonempty closed manifold. We reduce to the cotangent bundle via Weinstein's Lagrangian neighbourhood theorem.
\subsection*{Results related to Theorem \ref{thm:squeeze}}
Let $n\in\N_0$. The next result is a variant of Theorem \ref{thm:squeeze} for a bounded \emph{isotropic submanifold} of $\R^{2n}$ without boundary. It implies that every such submanifold is arbitrarily squeezable.
\begin{prop}[arbitrary squeezing for isotropic submanifold]\label{prop:isotr} Every bounded%
\footnote{as a subset of $\R^{2n}$} 
isotropic submanifold of $\R^{2n}$ without boundary ambiently symplectically embeds%
\footnote{See Definition \ref{defi:sympl}\reff{defi:sympl:emb}.} %
into every nonempty symplectic manifold of dimension $2n$.
\end{prop}
For a proof see Section \ref{sec:proof prop:isotr}. The result easily follows from the isotropic neighbourhood theorem and may therefore be well-known.
\begin{Rmk}[arbitrary squeezing for isotropic submanifold] Under the additional hypothesis that the dimension of the submanifold is strictly less than $n$, the statement of Proposition \ref{prop:isotr} follows from Theorem \ref{thm:squeeze}. To see this, we use the facts that every $m$-dimensional manifold is countably $m$-rectifiable and that it is critically negligible, if $m<n$.
\end{Rmk}

The next proposition implies that every bounded \emph{symplectic submanifold} of $\R^{2n}$ without boundary, of codimension at least 2, is arbitrarily squeezable.
\begin{prop}[arbitrary squeezing for certain symplectic submanifolds of $\R^{2n}$]\label{prop:sympl sq} A subset of $\R^{2n}$ ambiently symplectically embeds into every nonempty symplectic manifold of dimension $2n$, if it is one of the following:
\begin{enua}\item\label{prop:sympl sq:bdd}a bounded symplectic submanifold without boundary, whose codimension is at least 2
\item\label{prop:sympl sq:R2m}$\R^{2m}\x\{0\}$ %
\footnote{Our convention for $\omst$ is such that $\R^{2m}\x\{0\}\sub\R^{2n}$ is a symplectic submanifold.} %
with $m<n$
\end{enua}
\end{prop}
For a proof of this proposition see Section \ref{sec:proof prop:sympl}. The result easily follows from the isosymplectic embedding theorem of Gromov and the symplectic neighbourhood theorem. It is well-known, at least if the subset is as in \reff{prop:sympl sq:R2m}. In this case \footnote{with $m=2$ and $n=3$} a proof is outlined in F.~Schlenk's article \cite[p.~176]{SchlOld}. In our proof of Proposition \ref{prop:sympl sq} we fill in some details that are missing in \cite[p.~176]{SchlOld} and adapt the method of proof to situation \reff{prop:sympl sq:bdd}.
\begin{Rmk}Under the additional hypothesis that the dimension of the submanifold is less than $n$, the statement of Proposition \ref{prop:sympl sq} follows from Theorem \ref{thm:squeeze}.
\end{Rmk}

The next result shows that the $n$-negligibility condition in Theorem \ref{thm:squeeze} is sharp. (Compare to Remark \ref{rmk:negli}.) It is due to J.~Swoboda and F.~Ziltener, see \cite[Theorem 1]{SZsqueeze}.
\begin{thm}[nonsqueezable set of critical dimension]\label{thm:nonsqueezable} For every $n\geq2$ there exists a compact $n$-Hausdorff-dimensional subset of the closed ball in $\R^{2n}$ of radius $\sqrt2$ that does not injectively symplectically map into the open symplectic cylinder of radius 1. 
\end{thm}
Such a set may be constructed as the union of a closed Lagrangian submanifold and the image of a smooth map from $S^2$ to $\R^{2n}$, see the proof of \cite[Theorem 1]{SZsqueeze}. Theorem \ref{thm:nonsqueezable} sharpens the following theorem, which follows from M.~Gromov's and J.-C.~Sikorav's result \cite[Th\'eor\`eme 1]{Sik}%
\footnote{This result states that there does not exist any symplectomorphism of $\R^{2n}$ that maps the standard Lagrangian torus into the open standard symplectic cylinder.} 
or from D.~Hermann's result \cite[Theorem 1.12]{Her}, using the Extension after Restriction Principle. (See \cite[p.~8]{Schl}.) We denote by $\T^n=(S^1)^n$ the standard Lagrangian torus in $\R^{2n}$.
\begin{thm*}[Symplectic Hedgehog Theorem] For every $n\geq2$ the subset
\[[0,1]\cdot\T^n:=\big\{cx\,\big|\,c\in[0,1],\,x\in\T^n\big\}\]
of $\R^{2n}$ does not injectively symplectically map into the open symplectic cylinder of radius 1.
\end{thm*}
\begin{Rmk}This subset has Hausdorff dimension $n+1$.
\end{Rmk}
A result by J.~Swoboda and F.~Ziltener states that no product of odd-dimensional unit spheres of dimension at least 3 injectively symplectically maps into the standard symplectic cylinder, see \cite[Corollary 5]{SZCoiso}. This result provides further examples of small subsets that are not arbitrarily squeezable.
\subsection*{Further results related to Question \ref{q:small}}
In \cite[Proposition 3, Theorems 5 and 6]{SZHofer} J.~Swoboda and F.~Ziltener consider Question \ref{q:small} from a dynamical point of view, providing upper and lower bounds on the relative Hofer diameter of certain small subsets of a symplectic manifold.

\section{Proofs of the main results}\label{sec:proof main}
\subsection*{Proof of Theorem \ref{thm:displ}}
The proof of this result is based on the following lemma, whose proof captures the measure theoretic part of the argument. For every $A\sub\R^\ell$ and $v\in\R^\ell$ we denote
\[A+v:=\big\{x+v\,\big|\,x\in A\big\}.\]
\begin{lemma}[instantaneous affine displaceability of small sets]\label{le:displ} Let $\ell\in\N_0$, $m\in\{0,\ldots,\ell\}$, $A\sub\R^\ell$ be a countably $m$-rectifiable subset and $B\sub\R^\ell$ an $(\ell-m)$-negligible subset. There exists a vector $v_0\in\R^\ell$, such that for almost every $t\in\R$, we have
\begin{equation}\label{eq:A+tv0}(A+tv_0)\cap B=\emptyset.
\end{equation}
\end{lemma}

\begin{proof}[Proof of Theorem \ref{thm:displ}] We define $\ell:=2n$ and choose a vector $v_0\in\R^{2n}$ as in Lemma \ref{le:displ}. We denote by $\omst$ the standard symplectic form on $\R^{2n}$. We define the function
\[H:\R^{2n}\to\R,\qquad H(x):=\omst(v_0,x).\]
This function is linear. The corresponding Hamiltonian vector field is given by $X_H\const v_0$ and its flow by
\[\phi^t_H(x)=x+tv_0,\qquad\forall t\in\R,\,x\in\R^{2n}.\]
Therefore, the conclusion of Theorem \ref{thm:displ} follows from \eqref{eq:A+tv0}.\end{proof}

The proof of Lemma \ref{le:displ} is based on the following lemma from geometric measure theory. Let $(X,d),(X',d')$ be metric spaces and $p\in[1,\infty)$. 
\begin{defi*}[$p$-product metric]We define the $p$-product metric of $d$ and $d'$ to be the metric $\wt d_p$ on $\wt X:=X\x X'$ given by
\[\wt d_p(\wt x,\wt y):=\sqrt[p]{d(x,y)^p+d'(x',y')^p},\qquad\forall\wt x=(x,x'),\wt y=(y,y')\in\wt X.\]
\end{defi*}

\begin{lemma}[negligibility and product]\label{le:prod rect neg} Let $m,m'\in\N_0$, $A$ be a countably $m$-rectifiable metric space, $B$ an $m'$-negligible metric space, and $p\in[1,\infty)$. Then the Cartesian product $A\x B$ is $(m+m')$-negligible w.r.t.~the $p$-product metric.
\end{lemma}
The proof of this lemma is based on the following product property of the Hausdorff measure. 

\begin{lemma}[negligibility and product with $\R^m$]\label{le:prod Rm X} Let $m,m'\in\N_0$, $(X,d)$ a metric space that is $m'$-negligible, and $p\in[1,\infty)$. Then $\R^m\x X$ is $(m+m')$-negligible w.r.t.~the $p$-product metric.
\end{lemma}
\begin{proof}[Proof of Lemma \ref{le:prod Rm X}] This follows from \cite[2.10.45.~Theorem, p.~202]{Fed69}, using that the $p$-product metrics for different $p$s are equivalent.
\end{proof}

In the proof of Lemma \ref{le:prod rect neg} we will also use the following remark.
\begin{rmk}[negligibility and Lipschitz continuity]\label{rmk:im Lip} Let $X$ and $X'$ be metric spaces, $f:X\to X'$ a Lipschitz continuous map, and $s\in [0, \infty)$, such that $X$ is $s$-negligible. Then the image $f(X)$ is $s$-negligible.
\end{rmk}
For every metric space $(X,d)$ and $s\in[0,\infty)$ we denote by $\HH^s=\HH_d^s$ the $s$-dimensional outer Hausdorff measure on $(X,d)$.
\begin{proof}[Proof of Lemma \ref{le:prod rect neg}] By our hypothesis that $A$ is countably $m$-rectifiable, there exists a collection $\F$ as in Definition \ref{defi:count m rect} with $X=A$. Let $f\in\F$. We denote by $S$ the domain of $f$. By Lemma \ref{le:prod Rm X}, using $S\sub\R^m$ and our hypothesis that $B$ is $m'$-negligible, $S\x B$ is $(m+m')$-negligible w.r.t.~the $p$-product metric. Since $f$ is Lipschitz, the map $f\x\id$ is Lipschitz w.r.t.~the $p$-product metric on either side. Using Remark \ref{rmk:im Lip}, it follows that the metric space
\[\im(f)\x B=\im(f\x\id)\]
is $(m+m')$-negligible w.r.t.~the $p$-product metric. We have
\begin{align*}
\HH^{m+m'}(A\x B)&\leq\HH^{m+m'}\left(\left(\bigcup_{f\in\F}\im(f)\right)\x B\right)\tag{by \eqref{eqn_covering_A}}\\
&\leq\sum_{f\in\F}\HH^{m+m'}(\im(f)\x B)\tag{using $\left(\bigcup_{f\in\F}\im(f)\right)\x B=\bigcup_{f\in\F}\big(\im(f)\x B\big)$ and $\si$-subadditivity of $\HH^{m+m'}$}\\
&=\sum_{f\in\F}0\tag{using that $\im(f)\x B$ is $(m+m')$-negligible}\\
&=0.
\end{align*}
This proves Lemma \ref{le:prod rect neg}.
\end{proof}
In the proof of Lemma \ref{le:displ}, we will also use the following remark.
\begin{rmk}[Hausdorff measure]\label{rmk:Hausdorff Lebesgue} The $n$-dimensional outer Hausdorff measure in $\R^n$ is proportional to the outer Lebesgue measure, see for example \cite[Corollary 5.22, p.50]{AE09}.
\end{rmk}
	
\begin{proof}[Proof of Lemma \ref{le:displ}] We denote by $\lam^*$ the $\ell$-dimensional outer Lebesgue measure. By Lemma \ref{le:prod rect neg} the product $A\x B$ is $\ell$-negligible. We define the displacement vector map to be the function
\begin{equation}\label{eq:v}\vv:\R^\ell\times\R^\ell\to\R^\ell,\qquad\vv(x,y):=y-x.\end{equation}
This function is Lipschitz continuous. Therefore, by Remark \ref{rmk:im Lip}, the image $\vv(A\x B)$ is $\ell$-negligible. Hence by Remark \ref{rmk:Hausdorff Lebesgue}, we have
\begin{equation}\label{eq:LL ell}\lam^*\big(\vv(A\x B)\big)=0.\end{equation}
We define the map
\[\psi:(0,\infty)\x S^{\ell-1}\to\R^\ell,\qquad\psi(r,v):=rv.\]
Denoting by $\int_{S^{\ell-1}}dv$ the integral over $S^{\ell-1}$ w.r.t.~standard Riemannian metric, we have
\begin{align*}\int_{S^{\ell-1}}\left(\int_0^{\infty} \chi_{\vv(A\times B)}\circ\psi(r, v)r^{\ell-1}dr\right)dv&=\int_{\R^\ell\wo\{0\}}\chi_{\vv(A\x B)}(x)\,dx\tag{by Fubini's theorem and the Change of Variables theorem}\\
&=\lam^*\big(\vv(A\x B)\big)\\
&=0\tag{by \eqref{eq:LL ell}}
\end{align*}
It follows that for almost every $v\in S^{\ell-1}$, we have
\begin{equation}\label{eq:int 0 infty}
\int_0^{\infty}\chi_{\vv(A\times B)}(rv)r^{\ell-1}dr=0.
\end{equation}
Since the antipodal map on the sphere is volume-preserving, there exists $v_0\in S^{\ell-1}$, such that \eqref{eq:int 0 infty} holds for $v=v_0$ and $v=-v_0$. We choose such a $v_0$. Using equation \eqref{eq:int 0 infty} with $v=v_0$, we have that $tv_0\notin \vv(A\times B)$ for almost every $t=r\in(0,\infty)$. Using equation \eqref{eq:int 0 infty} with $v=-v_0$, we have that $-tv_0\notin \vv(A\times B)$ for almost every $t\in(0,\infty)$. Hence, the set
\begin{equation}\label{eq:S:=}
S:=\left\{t\in\R\,\big|\,tv_0\notin\vv(A\times B)\right\}
\end{equation}
has full (1-dimensional) Lebesgue measure. The conclusion of Lemma \ref{le:displ} is therefore a consequence of the following claim.
\begin{claim}\label{claim_displacement_detail} For every $t\in S$ we have
\[(A+tv_0)\cap B=\emptyset.\]
\end{claim}
\begin{proof}[Proof of Claim \ref{claim_displacement_detail}] Let $t\in\R$ be such that $(A+tv_0)\cap B\neq\emptyset$. Then there exist $a\in A$, $b\in B$ such that
\begin{align*}
&a+tv_0=b,\\
\textrm{i.e.,}\qquad&tv_0=b-a=\vv(a, b)\tag{by \eqref{eq:v}}
\end{align*}
Hence $tv_0\in\vv(A\times B)$. Therefore, by \eqref{eq:S:=}, $t\notin S$. The statement of Claim \ref{claim_displacement_detail} follows.
\end{proof}
This proves Lemma \ref{le:displ}.
\end{proof}
\subsection*{Proof of Theorem \ref{thm:squeeze}}
This proof is based on the following lemma. For every Lebesgue measurable subset $S\sub\R^2$ we denote by $|S|$ its area, i.e., (two-dimensional) Lebesgue measure.
\begin{lemma}[symplectically squeezing to almost half the size]\label{le:half} Let $n\in\N_0$, $Q,R\sub\R^2$ be nonempty open rectangles, such that $|Q|<2|R|$, $K\sub\R^{2n-2}$ a compact subset, $U\sub\R^{2n-2}$ an open subset containing $K$, and $A\sub Q\x K$ a countably $n$-rectifiable and $n$-negligible subset. Then $A$ injectively symplectically maps into $R\x U$.
\end{lemma}
	
\begin{proof}[Proof of Theorem \ref{thm:squeeze}] Let $A$ be a subset of $\R^{2n}$ as in the hypothesis of Theorem \ref{thm:squeeze}, and $R_1,\ldots,R_n\sub\R^2$ be nonempty open rectangles. We choose a collection $Q_i\sub\R^2$, $i=1,\ldots,n$, of open rectangles, such that the closure of $A$ is contained in $Q_1\x\cdots\x Q_n$. Lemma \ref{le:half}, applied sufficiently many times, implies that $A$ injectively symplectically maps into some compact subset of $R_1\x Q_2\x\cdots\x Q_n$. Intertwining the roles of the first and second factors in the product $\R^{2n}=\R^2\x\cdots\x\R^2$ and again applying Lemma \ref{le:half} sufficiently many times, it follows that $A$ injectively symplectically maps into some compact subset of $R_1\x R_2\x Q_3\x\cdots\x Q_n$. Continuing the same way, an induction argument shows that $A$ injectively symplectically maps into some compact subset of $R_1\x\cdots\x R_n$. The conclusion of Theorem \ref{thm:squeeze} follows.
\end{proof}

In the proof of Lemma \ref{le:half} we will use the following remark.
\begin{rmk}[symplectically embedding subsets of the plane]\label{rmk:square slit}\begin{enui}
\item\label{rmk:square slit:V de}Let $\de\in(0,1)$. We define 
\[V_\de:=\big((0,1)\x(-1,1)\big)\wo\big([\de,1)\x[0,\de]\big).\]
See Figure \ref{fig:V de proj A}. This is an open subset of $\R^2$ that is symplectomorphic to $V_{1-\de}$. This follows from the fact that there is a Hamiltonian diffeomorphism of $(0,1)\x(-1,1)$ that bijectively maps the set $[\de,1)\x[0,\de]$ to the set $\big[1-\de,1\big)\x[0,1-\de]$. %
\footnote{Alternatively, it follows from \cite[Theorem 1]{GS}.} 
\footnote{Such a map is e.g.~given by the Hamiltonian time-$t$ flow of the suitably cut off function $H(q,p):=qp$ for a suitable $t$.} 
\item\label{rmk:square slit:prod} Let $Q\sub\R^2$ be an open rectangle of area $|Q|\leq2-2\de$. There is an affine symplectomorphism of $\R^2$ that maps $Q$ into the rectangle $\big(0,1-\de\big)\x(-1,1)$, which is contained in $V_{1-\de}$. Using Remark \reff{rmk:square slit:V de}, it follows that there is a symplectic embedding $\theta$ of $Q$ into $V_\de$.
\end{enui}
\end{rmk}
	
\begin{proof}[Proof of Lemma \ref{le:half}]\setcounter{claim}{0} By some rescaling argument, we may assume \Wlog that $|Q|<2$ and $|R|>1$. We choose 
\begin{align}\label{eq:eps}\eps&\in\big(0,|R|-1\big),\\
\label{eq:de}\de&\in\left(0,\min\left\{\frac\eps2,1-\frac{|Q|}2\right\}\right).
\end{align}
By \eqref{eq:de} we have $|Q|<2-2\de$. Therefore, there is a symplectic embedding $\theta$ as in Remark \ref{rmk:square slit}\reff{rmk:square slit:prod}. The set $(\theta\x\id)(A)$ is contained in $V_\de\x K$. It is countably $n$-rectifiable and $n$-negligible, as the same holds for $A$ by hypothesis, and $\theta\x\id$ is locally Lipschitz.%
\footnote{Here we use Remark \ref{rmk:im Lip}.} 
Therefore, \Wlog we may assume that
\begin{equation}\label{eq:A sub}A\sub V_\de\x K.\end{equation}
We choose a smooth function $f:\R\to[0,1]$, such that 
\[f=\left\{\begin{array}{ll}
0,&\textrm{on }\left(-\infty,\displaystyle\frac\eps2\right],\\
1,&\textrm{on }[\eps,\infty).
\end{array}\right.\]
We define the non-linear shear
\begin{equation}\label{eq:psi}\psi:\R^2\to\R^2,\qquad\psi(q,p):=\big(q,p-f(q)\big),\end{equation}
see Figure \ref{fig:psi folded} on p.~\pageref{fig:psi folded}. Writing a point in $\R^{2n}=\R\x\R\x\R^{2n-2}$ as $(q,p,z)$, we define
\begin{eqnarray}\nn&A_+:=\left\{(q,p,z)\in A\,\big|\,p>0\right\},\qquad A_-:=\left\{(q,p,z)\in A\,\big|\,p\leq0\right\},&\\
\label{eq:A+'}&A_+':=(\psi\x\id)(A_+),&
\end{eqnarray}
see Figure \ref{fig:V de proj A} on p.~\pageref{fig:V de proj A}. We define
\begin{equation}\label{eq:S}S:=\big([0,1]\x[-1,0]\big)\cup\big([0,\eps]\x[0,1]\big),\end{equation}
see Figure \ref{fig:S}.%
\begin{figure}
\centering

	\def\svgwidth{1\columnwidth}
\begingroup%
  \makeatletter%
  \providecommand\color[2][]{%
    \errmessage{(Inkscape) Color is used for the text in Inkscape, but the package 'color.sty' is not loaded}%
    \renewcommand\color[2][]{}%
  }%
  \providecommand\transparent[1]{%
    \errmessage{(Inkscape) Transparency is used (non-zero) for the text in Inkscape, but the package 'transparent.sty' is not loaded}%
    \renewcommand\transparent[1]{}%
  }%
  \providecommand\rotatebox[2]{#2}%
  \newcommand*\fsize{\dimexpr\f@size pt\relax}%
  \newcommand*\lineheight[1]{\fontsize{\fsize}{#1\fsize}\selectfont}%
  \ifx\svgwidth\undefined%
    \setlength{\unitlength}{340.15748031bp}%
    \ifx\svgscale\undefined%
      \relax%
    \else%
      \setlength{\unitlength}{\unitlength * \real{\svgscale}}%
    \fi%
  \else%
    \setlength{\unitlength}{\svgwidth}%
  \fi%
  \global\let\svgwidth\undefined%
  \global\let\svgscale\undefined%
  \makeatother%
  \begin{picture}(1,0.79166667)%
    \lineheight{1}%
    \setlength\tabcolsep{0pt}%
    \put(0,0){\includegraphics[width=\unitlength,page=1]{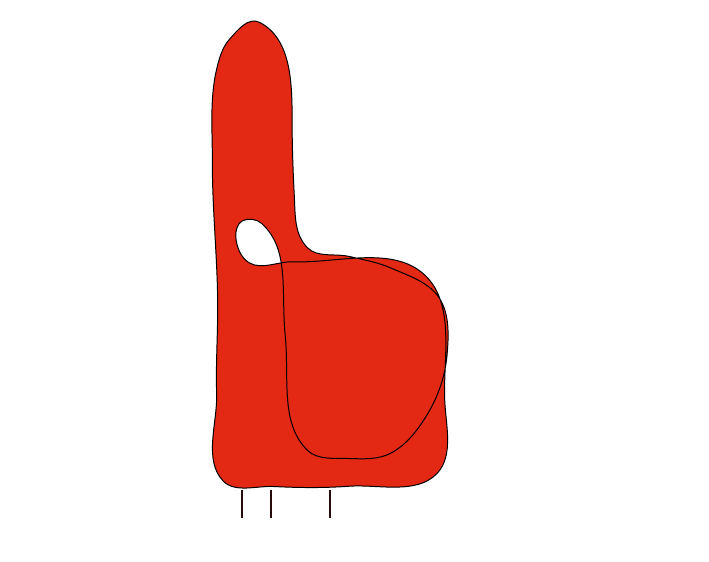}}%
    \put(0.32813534,0.0294832){\color[rgb]{0.00784314,0.00392157,0}\transparent{0.60583901}\makebox(0,0)[lt]{\lineheight{1.25}\smash{\begin{tabular}[t]{l}$\delta$\end{tabular}}}}%
    \put(0.37559155,0.02948322){\color[rgb]{0.00784314,0.00392157,0}\transparent{0.60583901}\makebox(0,0)[lt]{\lineheight{1.25}\smash{\begin{tabular}[t]{l}$\frac{\varepsilon}{2}$\end{tabular}}}}%
    \put(0.46063002,0.02948055){\color[rgb]{0.00784314,0.00392157,0}\transparent{0.60583901}\makebox(0,0)[lt]{\lineheight{1.25}\smash{\begin{tabular}[t]{l}$\varepsilon$\end{tabular}}}}%
    \put(0,0){\includegraphics[width=\unitlength,page=2]{squeezing3.pdf}}%
    \put(0.50246574,0.73057604){\color[rgb]{0,0,0}\makebox(0,0)[lt]{\lineheight{1.25}\smash{\begin{tabular}[t]{l}$S$\end{tabular}}}}%
  \end{picture}%
\endgroup%

\caption{The set $S$. It contains the projections of the folded set $A_+'=(\psi\x\id)(A_+)$ and of $A_-$ (both in red).}
\label{fig:S}
\end{figure}
Let $V$ be an open neighbourhood of $S$.
\begin{claim}\label{claim:phi} There exists a Hamiltonian diffeomorphism $\phi$ of $\R^{2n}$, such that
\begin{eqnarray}
\label{eq:phi=id}&\phi=\id\quad\textrm{on}\quad(-\infty,\de]\x\R^{2n-1},&\\
\label{eq:phi A+' cap}&\phi(A_+')\cap A_-=\emptyset,&\\
\label{eq:phi A+' sub}&\phi(A_+')\sub V\x U.&
\end{eqnarray}
\end{claim}
\begin{proof}[Proof of Claim \ref{claim:phi}] Our hypothesis that $A$ is countably $n$-rectifiable and the Lipschitz-property of $\psi$ imply that the set $A_+'$ is countably $n$-rectifiable. Our hypothesis that $A$ is $n$-negligible implies the same for $A_-$. Therefore, the hypotheses of Theorem \ref{thm:displ} are satisfied with $m:=n$ and $A,B$ replaced by $A_+',A_-$. Applying this theorem, there exists a linear function $H:\R^{2n}\to\R$, such that for almost every $t\in\R$, we have
\begin{equation}\label{eq:phi H t A+'}
\phi_H^t(A_+')\cap A_-=\emptyset.
\end{equation}
By \eqref{eq:de} we have $\de<\frac\eps2$. Hence there is a number $\eps'\in\left(\de,\frac\eps2\right)$. We choose a smooth function $\rho:\R\to\R$, such that 
\[\rho=\left\{\begin{array}{ll}
0,&\textrm{on }(-\infty,\de],\\
1,&\textrm{on }[\eps',\infty).
\end{array}\right.\]
We define the function
\[\wt H:\R^{2n}\to\R,\qquad\wt H(q,p,z):=\rho(q)H(q,p,z).\]
Since $H$ is linear, its Hamiltonian vector field is constant. Since $\wt H=H$ on $\left(\eps',\infty\right)\x\R^{2n-1}$ and $\eps'<\frac\eps2$, it follows that there exists a real number $t_0>0$, such that
\begin{equation}\label{eq:phi wt H t}\phi_{\wt H}^t=\phi_H^t\quad\textrm{on}\quad\left[\frac\eps2,\infty\right)\x\R^{2n-1},\qquad\forall t\in[0,t_0).
\end{equation}
We define
\begin{align}\label{eq:B0}B_0&:=A_+'\cap\big((-\infty,\de]\x\R^{2n-1}\big),\\
\label{eq:B1}B_1&:=A_+'\cap\left(\left(\de,\frac\eps2\right)\x\R^{2n-1}\right),\\
\label{eq:B2}B_2&:=A_+'\cap\left(\left[\frac\eps2,\infty\right)\x\R^{2n-1}\right).
\end{align}
See Figure \ref{fig:B}. %
\begin{figure}
\centering

	\def\svgwidth{1\columnwidth}
	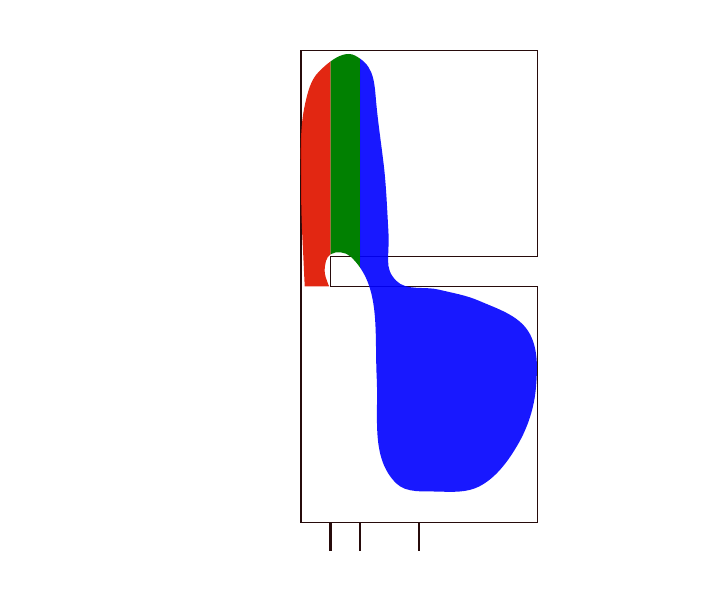

\caption{The projections of the sets $B_0,B_1,B_2$ to the first factor in $\R^{2n}=\R^2\x\cdots\x\R^2$.}
\label{fig:B}
\end{figure}
By \eqref{eq:psi} and the fact $f\big|_{\left(-\infty,\frac\eps2\right]}=0$, we have $\psi=\id$ on $\left(-\infty,\frac\eps2\right]\x\R$. It follows that 
\begin{align}\label{eq:B0 A+}B_0&=A_+\cap\big((-\infty,\de]\x\R^{2n-1}\big),\\
\label{eq:B1 A+}B_1&=A_+\cap\left(\left(\de,\frac\eps2\right)\x\R^{2n-1}\right).\end{align}
By \eqref{eq:A sub}, the right hand side of \eqref{eq:B1 A+} is contained in the compact set $C:=\left[\de,\frac\eps2\right]\x[\de,1]\x K$, and hence
\begin{equation}\label{eq:B1 sub C}B_1\sub C.
\end{equation}
The set $A_-$ is contained in the closed set $\R\x(-\infty,0]\x\R^{2n-2}$, which is disjoint from $C$. Using \eqref{eq:B1 sub C}, it follows that there exists a real number $t_1>0$, such that
\begin{equation}\label{eq:phi wt H t B1}\phi_{\wt H}^t(B_1)\cap A_-=\emptyset,\qquad\forall t\in(0,t_1).
\end{equation}
We have 
\[A_+'\sub S\x K,\]
where $S$ is defined as in \eqref{eq:S}. The set $S\x K$ is compact and contained in the open set $V\x U$. It follows that there exists a real number $t_2>0$, such that
\begin{equation}\label{eq:V U}\phi_{\wt H}^t(A_+')\sub V\x U,\qquad\forall t\in(0,t_2).
\end{equation}
We choose $t\in\big(0,\min\{t_0,t_1,t_2\}\big)$, such that \eqref{eq:phi H t A+'} is satisfied. We define
\[\phi:=\phi_{\wt H}^t.\]
Claim \ref{claim:phi} follows from the next claim.
\begin{claim}\label{claim:phi cond} The map $\phi$ satisfies conditions (\ref{eq:phi=id},\ref{eq:phi A+' cap},\ref{eq:phi A+' sub}).
\end{claim}
\begin{pf}[Proof of Claim \ref{claim:phi cond}] Since $\wt H=0$ on $(-\infty,\de]\x\R^{2n-1}$, we have $\phi=\id$ on $(-\infty,\de]\x\R^{2n-1}$. Hence $\phi$ satisfies condition \eqref{eq:phi=id}.

We check condition \reff{eq:phi A+' cap}. By (\ref{eq:B0},\ref{eq:B1},\ref{eq:B2}) we have
\begin{equation}\label{eq:A+'=}A_+'=B_0\cup B_1\cup B_2.\end{equation}
Conditions (\ref{eq:phi=id},\ref{eq:B0 A+}) and the fact $A_+\cap A_-=\emptyset$ imply that
\begin{equation}\label{eq:phi B0}\phi(B_0)\cap A_-=\emptyset.
\end{equation}
By (\ref{eq:B2},\ref{eq:phi wt H t},\ref{eq:phi H t A+'}) we have
\[\phi(B_2)\cap A_-=\emptyset.\]
Combining this with (\ref{eq:A+'=},\ref{eq:phi B0},\ref{eq:phi wt H t B1}), it follows that
\[\phi(A_+')\cap A_-=\emptyset.\]
Therefore, $\phi$ satisfies condition \eqref{eq:phi A+' cap}. By \eqref{eq:V U}, it satisfies condition \eqref{eq:phi A+' sub}.
\end{pf}
This proves Claim \ref{claim:phi cond} and therefore Claim \ref{claim:phi}.
\end{proof}
It follows from (\ref{eq:S},\ref{eq:eps}) that there exists an open neighbourhood $V$ of $S$ and a symplectic embedding
\begin{equation}\label{eq:chi}\chi:V\to R.\end{equation}
We choose $\phi$ as in Claim \ref{claim:phi}. We denote
\[V_\de^+:=V_\de\cap\big(\R\x(0,\infty)\big),\qquad V_\de^-:=V_\de\cap\big(\R\x(-\infty,0]\big)=(0,1)\x(-1,0].\]
We define the map
\[\wt\phi:V_\de\x\R^{2n-2}\to\R^{2n},\qquad\wt\phi:=\left\{\begin{array}{ll}
(\chi\x\id)\circ\phi\circ(\psi\x\id),&\textrm{on }V_\de^+,\\
\chi\x\id,&\textrm{on }V_\de^-.
\end{array}\right.\]
It follows from \eqref{eq:phi=id} and the fact $\psi\big|_{\left(-\infty,\frac\eps2\right]\x\R}=\id$ that the map $\wt\phi$ is smooth. It is symplectic. By (\ref{eq:A+'},\ref{eq:phi A+' sub},\ref{eq:chi}), $\wt\phi$ maps $A$ into $R\x U$. It follows from \eqref{eq:phi A+' cap} that the restriction of $\wt\phi$ to $A$ is injective. Hence this restriction is an injective symplectic map from $A$ into $R\x U$. This concludes the proof of Lemma \ref{le:half}.
\end{proof}

\subsection*{Proof of Corollary \ref{cor:inf d}}
\label{proof:cor:inf d} The next claim implies that $\Ddisp\geq n$. Let $D\in[0,n)$ and $A\sub\R^{2n}$ be a $D$-Hausdorff dimensional, bounded, and countably $\lceil D\rceil$-rectifiable subset. 
\setcounter{claim}{0}
\begin{claim}\label{claim:wt H} For every $\eps>0$ there exists $\wt H\in C^\infty(\R^{2n},\R)$, such that
\begin{align}\label{eq:phi wt H 1}\phi_{\wt H}^1(A)\cap A&=\emptyset,\\
\label{eq:Vert wt H Vert}\Vert\wt H\Vert&<\eps.
\end{align}
\end{claim}
\begin{proof}[Proof of Claim \ref{claim:wt H}] By Definition \ref{def:neg Hd dim} $A$ is $s$-negligible for some $s\in[D,n)$. Hence $A$ is $n$-negligible. Since $A$ is $\lceil D\rceil$-rectifiable, it is $n$-rectifiable. Therefore, it satisfies the hypothesis of Corollary \ref{cor:displ}. Applying that corollary, there exists a function $H$ as in the conclusion of that theorem. Since $A$ is bounded, the sets $\BAR A$ and hence $K:=\bigcup_{t\in[0,1]}\phi_H^t(\BAR A)$ are compact. We choose a function $\rho\in C^\infty(\R^{2n},\R)$ that equals 1 on some compact neighbourhood of $K$ and has compact support. By \eqref{eq:phi t H} there exists $t\in[0,1]$, such that
\begin{align}\label{eq:phi H t A}\phi_H^t(A)\cap A&=\emptyset,\\
\nn t&<\frac\eps{\Vert\rho H\Vert}.\end{align}
We define $\wt H:=t\rho H$. This function satisfies \eqref{eq:Vert wt H Vert}. We have
\[\phi_{\wt H}^1=\phi_{\rho H}^t,\qquad\phi_{\rho H}^t=\phi_H^t\quad\textrm{on}\quad\BAR A.\]
Using \eqref{eq:phi H t A}, it follows that $\phi_{\wt H}^1(A)\cap A=\emptyset$, i.e., \eqref{eq:phi wt H 1} holds. This proves Claim \ref{claim:wt H}.
\end{proof}
Since $\left(\R^{2n},\omst\right)$ is geometrically bounded, the displacement energy of every nonempty closed Lagrangian submanifold of $\R^{2n}$ is positive, see e.g.~\cite{Che}. It follows that $\Ddisp\leq n$. Therefore, by Claim \ref{claim:wt H}, we have $\Ddisp=n$. This proves Corollary \ref{cor:inf d}. \qed

\section{Proofs of three of the related results}\label{sec:proof prop}
\subsection{Proof of  Proposition \ref{prop:Lag not C1 inst displ} (instantaneous non-displaceability of Lagrangian submanifold)}\label{sec:proof prop:Lag}
The proof of Proposition \ref{prop:Lag not C1 inst displ} is based on a result of M.~Gromov, which states that a Hamiltonian isotopy does not displace the zero-section of the cotangent bundle of a nonempty closed manifold. We reduce to the cotangent bundle via Weinstein's Lagrangian neighbourhood theorem.

This reduction is based on a \emph{$C^1$-close symplectic isotopy lemma}, which is a version for a noncompact manifold of the following fact: On a closed symplectic manifold $(M,\om)$ every symplectomorphism that is $C^1$-close to the identity is the time-1-restriction of a symplectic isotopy that stays $C^1$-close to the identity. This follows from Weinstein's Lagrangian neighbourhood theorem applied to the diagonal in $M\x M$.

The next remark is a reformulated version of Gromov's result.
\begin{rmk}[Hamiltonian isotopy in cotangent bundle]\label{rmk:Ham isot} Let $Q$ be a closed manifold. We equip $T^*Q$ with the canonical symplectic form $\omcan$ and denote by $0^Q\sub T^*Q$ the 0-section. Let $W$ be an open neighbourhood of $0^Q$ and $\psi\in C^\infty\big([0,1]\x W,T^*Q\big)$, such that 
\[\psi_0(0^Q)=0^Q,\]
and denoting
\begin{equation}\label{eq:Xt}X_t:=\left(d\psi_t\right)^{-1}\left.\frac d{ds}\right|_{s=t}\psi_s,\end{equation}
there exists a function $H\in C^\infty\big([0,1]\x W,\R\big)$ satisfying
\begin{equation}\label{eq:d psi t}\iota_{X_t}\omcan=dH_t,\qquad\forall t\in[0,1].\end{equation}
Then we have
\begin{equation}\label{eq:0Q}0^Q\cap\psi_1(0^Q)\neq\emptyset.\end{equation}
To see this, we cut off the function $(t,x)\mapsto\psi^t_*H_t:=H_t\circ\psi_t^{-1}$ outside of some open neighbourhood of the compact set $\bigcup_{t\in[0,1]}\psi_t(Q)$, obtaining a compactly supported function $\wt H\in C^\infty\big([0,1]\x T^*Q,\R\big)$. By \eqref{eq:Xt} $\psi$ is the flow of the time-dependent vector field $\left(\psi^t_*X_t\right)_t$. By \eqref{eq:d psi t} we have $\psi^t_*X_t=X_{\psi^t_*H_t}$. It follows that
\[\phi_{\wt H}^t=\psi_t\quad\textrm{on}\quad Q.\]
The condition \eqref{eq:0Q} now follows from a result by M.~Gromov, see \cite[Theorem 11.3.10]{MSIntro}.
\end{rmk}

To formulate the $C^1$-close symplectic isotopy lemma, we denote for each subset $A$ of a topological space $X$ by $\Int A$ the interior of $A$ in $X$. For each subset $A$ of a set $S$ we denote by $i_A:A\to S$ the inclusion map. For each pair of symplectic manifolds $(M,\om),(M',\om')$ we denote
\[\Embs(M',M):=\big\{\textrm{symplectic embedding }M'\to M\big\}.\]
Let $(M,\om)$ be a symplectic manifold without boundary, and $K_0,K$ be compact submanifolds of $M$ 
of dimension $\dim M$ (possibly with boundary), such that
\begin{equation}\label{eq:K0 sub Int K}K_0\sub\Int K.\end{equation}
Let $\U_0$ be a $C^1$-neighbourhood of $i_{K_0}$ in $\Embs(K_0,M)$.
\begin{lemma}[$C^1$-close symplectic isotopy]\label{le:sympl isot} There exists a $C^1$-neighbourhood $\U$ of $i_K$ in $\Embs(K,M)$, such that for every $\phi\in\U$ there exists a map $\psi\in C^\infty\big([0,1]\x K_0,M\big)$ satisfying
\begin{align}
\label{eq:psi t}\psi_t:=\psi(t,\cdot)&\in\U_0,\qquad\forall t\in[0,1],\\
\label{eq:psi 0}\psi_0&=i_{K_0},\\
\label{eq:psi 1}\psi_1&=\phi|_{K_0}.
\end{align}
\end{lemma}
For a proof of this lemma see p.~\pageref{proof:le:sympl isot}.

We are now ready for the proof of Proposition \ref{prop:Lag not C1 inst displ}.
\begin{proof}[Proof of Proposition \ref{prop:Lag not C1 inst displ} (p.~\pageref{prop:Lag not C1 inst displ})]\label{proof:prop:Lag not C1 inst displ} Let $L$ be a nonempty closed Lagrangian submanifold of $\R^{2n}$. We choose compact submanifolds $K_0,K$ of $\R^{2n}$ of dimension $2n$, such that the first de Rham cohomology of $K_0$ vanishes and
\begin{align}\label{eq:L sub Int K0}L&\sub\Int K_0,\\
\label{eq:K0 Int K}K_0&\sub\Int K.
\end{align}
By Weinstein's Lagrangian neighbourhood theorem there exist open neighbourhoods $U$ of $L$ in $\Int K_0$ and $V$ of the 0-section in $T^*L$ and a symplectomorphism $U\to V$ that restricts to the canonical inclusion of $L$ in $T^*L$. We define
\begin{equation}\label{eq:U0}\U_0:=\big\{\phi_0\in\Embs\big(K_0,\R^{2n}\big)\,\big|\,\phi_0(L)\sub U\big\}.\end{equation}
By \eqref{eq:L sub Int K0} the condition $\phi_0(L)\sub U$ makes sense, and therefore, the set $\U_0$ is well-defined. It is a compact-open neighbourhood of $i_{K_0}$ in $\Embs\big(K_0,\R^{2n}\big)$. Using \eqref{eq:K0 Int K}, Lemma \ref{le:sympl isot} therefore implies that there exists a set $\U$ as in that lemma. We denote by $\Ham(\R^{2n})$ the group of Hamiltonian diffeomorphism of $\R^{2n}$ w.r.t.~$\omst$ and define
\[\U':=\big\{\phi'\in\Ham(\R^{2n})\,\big|\,\phi'|_K\in\U\big\}.\]
This is a weak $C^1$-neighbourhood of $\id$ in $\Ham(\R^{2n})$. Proposition \ref{prop:Lag not C1 inst displ} therefore follows from the next claim.
\setcounter{claim}{0}
\begin{claim}\label{claim:phi'}
For every $\phi'\in\U'$ we have
\[L\cap\phi'(L)\neq\emptyset.\]
\end{claim}
\begin{pf}[Proof of Claim \ref{claim:phi'}] We define
\begin{equation}\label{eq:phi}\phi:=\phi'|_K.\end{equation}
Since $\phi\in\U$, by the statement of Lemma \ref{le:sympl isot} there exists a map $\psi\in C^\infty\big([0,1]\x K_0,M\big)$ satisfying (\ref{eq:psi t},\ref{eq:psi 0},\ref{eq:psi 1}). For every $t\in[0,1]$ we define the vector field $X_t$ on $K_0$ by
\[X_t:=\left(d\psi_t\right)^{-1}\left.\frac d{ds}\right|_{s=t}\psi_s:K_0\to TM.\]
Let $t\in[0,1]$. By \reff{eq:psi t}, $\psi_s$ is symplectic for every $s$. Since
\[X_t=\psi_t^*\left(\left(\left.\frac d{ds}\right|_{s=t}\psi_s\right)\circ\psi_t^{-1}\right),\]
it follows that $X_t$ is symplectic. Since the first de Rham cohomology of $K_0$ vanishes, it follows that
\[\al_t:=\iota_{X_t}\omst\]
is exact. Since the map $(t,x)\mapsto(\al_t)_x$ is smooth, it follows that there exists a smooth function $H:[0,1]\x\Int K_0\to\R$, such that
\[dH_t=\al_t|\Int K_0,\qquad\forall t.\]
Here we used \cite[Theorem A.1, p.~475, and Remark A.3(i), p.~479]{AG}. By (\ref{eq:psi t},\ref{eq:U0}) we have
\[\psi_t(L)\sub U,\qquad\forall t.\]
By \reff{eq:psi 0} we have
\[\psi_0(L)=L.\]
Since $U$ is a Weinstein neighbourhood of $L$, it therefore follows from Remark \ref{rmk:Ham isot} that
\begin{align*}\emptyset&\neq L\cap\psi_1(L)\\
&=L\cap\phi(L)\qquad\textrm{(by \reff{eq:psi 1})}\\
&=L\cap\phi'(L)\qquad\textrm{(by \eqref{eq:phi}).}
\end{align*}
This proves Claim \ref{claim:phi'} \end{pf} and completes the proof of Proposition \ref{prop:Lag not C1 inst displ}.
\end{proof}

Lemma \ref{le:sympl isot} is a consequence of the following lemma. For every manifold $Q$ we denote by $0_Q:Q\to T^*Q$ the canonical inclusion as the 0-section. Let $(M,\om)$ be a symplectic manifold without boundary and $K_0,Q,K$ be compact submanifolds of $M$ of dimension $\dim M$, such that
\begin{equation}\label{eq:K0 sub Int Q}K_0\sub\Int Q,\qquad Q\sub\Int K.\end{equation}
Let $\U_0$ be a $C^1$-neighbourhood of $i_{K_0}$ in $\Embs(K_0,M)$.
\begin{lemma}[correspondence between symplectic embeddings and closed 1-forms]\label{le:sympl emb closed 1 form} There exist sets $\U,\V$ and maps
\[\Phi:\U\to\V,\qquad\Psi:\V\to\U_0\]
with the following properties:
\begin{enua}
\item\label{le:sympl emb closed 1 form:U}$\U$ is a $C^1$-neighbourhood of $i_K$ in $\Embs(K,M)$.
\item\label{le:sympl emb closed 1 form:V convex}$\V$ is a convex subset of $\big\{\al\in\Om^1(Q)\,\big|\,d\al=0\big\}$.%
\footnote{The proof of the lemma shows that $\V$ can be chosen to be a $C^1$-neighbourhood of $0_Q$. However, we will only use that $\V$ contains $0_Q$.}
\item\label{le:sympl emb closed 1 form:0Q in V}$0_Q\in\V$
\item\label{le:sympl emb closed 1 form:0Q}$\Psi(0_Q)=i_{K_0}$
\item\label{le:sympl emb closed 1 form:Psi circ Phi}$\Psi\circ\Phi=$ restriction to $K_0$
\item\label{le:sympl emb closed 1 form:al}Let $\al\in C^\infty\big([0,1]\x Q,T^*Q\big)$ be a map satisfying $\al_t:=\al(t,\cdot)\in\V$, for every $t\in[0,1]$. Then the map
\[[0,1]\x K_0\ni(t,x)\mapsto\Psi(\al_t)(x)\in M\]
is smooth.
\end{enua}
\end{lemma}
\begin{proof}[Proof of Lemma \ref{le:sympl emb closed 1 form}] This follows from an argument involving the following ingredients:
\begin{itemize}
\item Weinstein's Lagrangian neighbourhood theorem applied with the diagonal in $M\x M$
\item The set of $C^1$-embeddings of a compact manifold in a boundaryless manifold is $C^1$-open in the set of all $C^1$-maps.
\item Composition of $C^1$-maps is continuous w.r.t.~the weak $C^1$-topologies.
\item Inversion of $C^1$-embeddings between $C^1$-manifolds without boundary is continuous w.r.t.~the weak $C^1$-topologies.%
\footnote{To make sense of the inversion map, we need to restrict the inverted embedding to a fixed submanifold of the target manifold.}
\end{itemize}
\end{proof}

\begin{proof}[Proof of Lemma \ref{le:sympl isot}]\label{proof:le:sympl isot} By \eqref{eq:K0 sub Int K} there exists a compact submanifold $Q$ of $M$ of dimension $\dim M$, such that \eqref{eq:K0 sub Int Q} holds. We choose $\U,\V,\Phi,\Psi$ as in Lemma \ref{le:sympl emb closed 1 form}. By \reff{le:sympl emb closed 1 form:U} $\U$ is a $C^1$-neighbourhood of $i_K$. Let $\phi\in\U$. We define
\[\psi:[0,1]\x K_0\to M,\qquad\psi(t,x):=\Psi\big(t\Phi(\phi)\big)(x).\]
The fact $\im(\Phi)\sub\V$ and (\ref{le:sympl emb closed 1 form:0Q in V},\ref{le:sympl emb closed 1 form:V convex}) imply that $t\Phi(\phi)\in\V$, for every $t\in[0,1]$. Hence $\psi$ is well-defined. By \reff{le:sympl emb closed 1 form:al} it is smooth. Since $\Psi$ takes values in $\U_0$, condition \eqref{eq:psi t} is satisfied. Condition \reff{le:sympl emb closed 1 form:0Q} implies \eqref{eq:psi 0}. Condition \reff{le:sympl emb closed 1 form:Psi circ Phi} implies that \eqref{eq:psi 1} holds. Hence $\psi$ has the desired properties. This proves Lemma \ref{le:sympl isot}.
\end{proof}

\subsection{Proof of Proposition \ref{prop:isotr} (arbitrary squeezing for isotropic submanifold)}\label{sec:proof prop:isotr}
Let $N$ be a bounded isotropic submanifold of $\R^{2n}$ and $c\in(0,\infty)$. The rescaled set $cN:=\big\{cx\,\big|\,x\in N\big\}$ is an isotropic submanifold of $\R^{2n}$. There is a canonical symplectic vector bundle isomorphism $\Phi$ from the symplectic quotient bundle of the symplectic complement bundle of $TN$ to the corresponding bundle for $cN$, such that $\Phi$ covers the map $f:N\to cN$, $f(x):=cx$. Therefore, by the isotropic neighbourhood theorem, there are open neighbourhoods $U,U'$ of $N,cN$ and a symplectomorphism between $U$ and $U'$ that restricts to $c$ times the identity on $N$. Using the hypothesis that $N$ is bounded, it follows that $N$ ambiently symplectically embeds into every open neighbourhood of the origin in $\R^{2n}$. Using Darboux's theorem, it follows that $N$ ambiently symplectically embeds into every nonempty symplectic manifold of dimension $2n$. This proves Proposition \ref{prop:isotr}. \qed
\subsection{Proof of  Proposition \ref{prop:sympl sq} (arbitrary squeezing for certain symplectic submanifolds of $\R^{2n}$)}\label{sec:proof prop:sympl}
In the proof of Proposition \ref{prop:sympl sq} we will use the following lemma. Let $X$ be a manifold without boundary. We call two (smooth) symplectic vector bundles $(E_0,\om_0)$ and $(E_1,\om_1)$ over $X$ \emph{strongly isomorphic} iff there exists a symplectic isomorphism between them that covers the identity on $X$. In this case we write
\[(E_0,\om_0)\iso(E_1,\om_1).\]
For $j=0,1$ we define the map
\begin{equation}\label{eq:ij}i_j:X\to[0,1]\x X,\qquad i_j(x):=(j,x).\end{equation}
\begin{lemma}[symplectic vector bundle]\label{le:sympl vect bdl} Let $(E,\om)$ be a symplectic vector bundle over $[0,1]\x X$. We have
\begin{equation}\label{eq:i0 * E om}i_0^*(E,\om)\iso i_1^*(E,\om).\end{equation}
\end{lemma}
\begin{proof}[Proof of Lemma \ref{le:sympl vect bdl}] This follows from an analogous statement for the symplectic frame bundle $P$ of $(E,\om)$. That statement follows from an argument involving the holonomy of a connection 1-form on $P$ along the path $[0,1]\ni t\mapsto(t,x)\in[0,1]\x X$, for every $x\in X$.
\end{proof}
For every symplectic vector space $(V,\Om)$ and every linear subspace $W\sub V$ we denote the symplectic (orthogonal) complement of $W$ by
\[W^\Om:=\big\{v\in V\,\big|\,\Om(v,w)=0,\,\forall w\in W\big\}.\]
\begin{proof}[Proof of Proposition \ref{prop:sympl sq}] Let $N$ be a symplectic submanifold of $\R^{2n}$ as in \reff{prop:sympl sq:bdd}. We denote by $\omst$ the standard symplectic form on $\R^{2n}$ and by $i:N\inj\R^{2n}$ the inclusion. Let $U$ be an open neighborhood of the origin in $\R^{2n}$. Since $N$ is bounded, there exists a $c\in(0,\infty)$, such that the image of the map
\[f_0:=ci:N\to\R^{2n}\]
is contained in $U$. Since the forms $f_0^*(\omst|_U)$ and $i^*\omst$ are both exact, they are cohomologous. We define the map
\begin{equation}\label{eq:F}F:[0,1]\x TN\to TU,\qquad F(t,(x,v)):=\big(cx,\big((1-t)c+t\big)v\big),\end{equation}
where we canonically identify $T_x\R^{2n}=\R^{2n}=T_{cx}\R^{2n}$. For every $t\in [0, 1]$, $F_t$ is a fiberwise injective vector bundle morphism over $f_0$. Moreover, $F_0=df_0$ and $F_1$ is symplectic%
\footnote{w.r.t.~$i^*\omst$ and $\omst|_U$}
. It therefore follows from Gromov's isosymplectic embedding theorem that there exists an isotopy
\[f_t:N\to U,\qquad t\in[0,1]\]
starting at $f_0$, such that $f_1$ is symplectic, a (smooth) homotopy $g_t:N\to U$ ($t\in[0,1]$), and a homotopy of symplectic vector bundle morphisms $G_t:TN\to TU$ ($t\in[0,1]$) covering $(g_t)_{t\in[0,1]}$, such that
\[G_0=F_1,\qquad G_1=df_1.\]
(See \cite[20.1.1]{CEM24} or \cite[p.~335-336]{Gro86}.%
\footnote{Here in the case $\dim N>0$ we used that $N$ is open, i.e., none of its connected components is closed. This follows from exactness of $i^*\omst$.} 
) We define the symplectic vector bundle $(E,\om)$ over $[0,1]\x N$ by
\begin{equation}\label{eq:E t x}E_{(t,x)}:=\left(G_t(T_xN)\right)^{\omst_{g_t(x)}},\qquad\om_{(t,x)}:=\omst|_{E_{(t,x)}}.\end{equation}
We define $i_j$ as in \eqref{eq:ij} with $X:=N$. By Lemma \ref{le:sympl vect bdl} condition \eqref{eq:i0 * E om} holds. Let $x\in N$. We have 
\begin{align*}T_xN&=F_1(T_xN)\qquad\textrm{(by \eqref{eq:F})}\\
&=G_0(T_xN)\qquad\textrm{(since $G_0=F_1$).}
\end{align*}
It follows that 
\begin{align*}T_xN^{\omst_x}&=\left(G_0(T_xN)\right)^{\omst_{cx}}\\
&=E_{(0,x)}\qquad\textrm{(by \eqref{eq:E t x}),}
\end{align*}
\[\omst=\om\quad\textrm{on}\quad T_xN^{\omst_x}.\]
Denoting
\[N':=f_1(N),\]
it follows that
\begin{align*}\left(TN^{\omst},\omst|_{TN^{\omst}}\right)&=i_0^*(E,\om)\\
&\iso i_1^*(E,\om)\qquad\textrm{(by \eqref{eq:i0 * E om})}\\
&=f_1^*\left(T{N'}^{\omst},\omst|_{T{N'}^{\omst}}\right)\,\textrm{(using \eqref{eq:E t x} and $G_1=df_1$).}
\end{align*}
Since $f_1$ is a symplectic embedding, the symplectic neighbourhood theorem therefore implies that there exist open neighbourhoods $U_0$ of $N$ in $\R^{2n}$ and $U_1$ of $N'$ in $U$, and a symplectomorphism from $U_0$ to $U_1$ that restricts to $f_1$ on $N$. Using Darboux's theorem, it follows that $N$ ambiently symplectically embeds into every nonempty symplectic manifold of dimension $2n$. This proves the statement of Proposition \ref{prop:sympl sq} for a subset $N$ as in \reff{prop:sympl sq:bdd}.

Let now $N:=\R^{2m}\x\{0\}$ with $m<n$, as in \reff{prop:sympl sq:R2m}. Let $U$ be an open neighbourhood of 0 in $\R^{2n}$. We choose an embedding $f_0:N\to U$, such that $df_0(0)=\id$. We define the map
\[F:[0,1]\x TN\to TU,\qquad F(t,(x,v)):=\big(f_0(x),df_0\big((1-t)x\big)(v)\big),\]
where we canonically identify $T_x\R^{2n}=\R^{2n}=T_{f_0(x)}\R^{2n}$. By an argument as above, Gromov's isosymplectic embedding theorem implies that there exists an isotopy
\[f_t:N\to U,\qquad t\in[0,1]\]
starting at $f_0$, such that $f_1$ is symplectic. We denote
\[N':=f_1(N).\]
The symplectic vector bundle $\left(TN^{\omst},\omst|_{TN^{\omst}}\right)$ is trivial. Since $N'$ is smoothly contractible, it follows from Lemma \ref{le:sympl vect bdl} that the symplectic vector bundle $\left(T{N'}^{\omst},\omst|_{T{N'}^{\omst}}\right)$ is trivializable. It follows that 
\[TN^{\omst}\iso f_1^*T{N'}^{\omst}.\]
Since $f_1$ is a symplectic embedding, by an argument as above, the symplectic neighbourhood theorem and Darboux's theorem therefore imply that $N$ ambiently symplectically embeds into every nonempty symplectic manifold of dimension $2n$. This proves the statement of Proposition \ref{prop:sympl sq} for $N:=\R^{2m}\x\{0\}$ with $m<n$, as in \reff{prop:sympl sq:R2m}. This concludes the proof of this proposition.
\end{proof}

\begin{Rmks}[proof of Proposition \ref{prop:sympl sq}]\begin{itemize}\item Our proof is based on the method explained in \cite[p.~176]{SchlOld} for the symplectic submanifold $\R^4\x\{0\}$ of $\R^6$.
\item The method rests on Gromov's isosymplectic embedding theorem, which is a version of the h-principle. See \cite{CEM24}.
\end{itemize}
\end{Rmks}

\appendix

\section{Characterization of countable $m$-rectifiability, injectivity of restriction of locally injective map}\label{sec:count rect loc inj}
We prove the characterization of countable $m$-rectifiability provided by Lemma \ref{le:count rect}.
\begin{proof}[Proof of Lemma \ref{le:count rect}]\label{proof:le:count rect}``\reff{le:count rect:def}$\then$\reff{le:count rect:bdd}'' follows by choosing a countable exhausting collection of bounded subsets of $\R^m$ and considering the set of all restrictions of functions in $\F$ to these subsets.

To prove ``\reff{le:count rect:bdd}$\then$\reff{le:count rect:loc Lip}'', we assume that \eqref{le:count rect:bdd} holds. We choose a set $\F$ as in Definition \ref{defi:count m rect}, such that the domain of each function in $\F$ is bounded. We denote by $\dE$ the Euclidean distance function on $\R^m$. We may assume \Wlog that\footnote{Here $\dom(f)$ denotes the domain of $f$.}
\[\forall f\neq f'\in\F,\,y\in\dom(f),\,y'\in\dom(f'):\quad\dE(y,y')\geq1.\]
To see this, we replace each function $f\in\F$ by $f(\cdot-v_f)$, where $v_f\in\R^m$ is a suitable vector in $\R^m$. Here we use that $\F$ is countable. We define
\[S:=\bigcup_{f\in\F}\dom(f),\qquad F:=\bigcup_{f\in\F}f:S\to X.\]
$S$ is a subset of $\R^m$, and $F$ is surjective and locally Lipschitz. Hence \reff{le:count rect:loc Lip} holds. This proves ``\reff{le:count rect:bdd}$\then$\reff{le:count rect:loc Lip}''.

To prove ``\reff{le:count rect:loc Lip}$\then$\reff{le:count rect:def}'', we assume that \reff{le:count rect:loc Lip} holds. We choose 

a map $F$ as in \reff{le:count rect:loc Lip} and a countable base $\U$ for the topology of $\R^m$. The set
\[\F:=\big\{F|_{S\cap U}\,\big|\,U\in\U:\,F|_{S\cap U}\textrm{ is Lipschitz}\big\}\]
is countable, consists of Lipschitz maps, and satisfies
\[\bigcup_{f\in\F}\im(f)=\im(F)=X.\]
It follows that \eqref{le:count rect:def} holds. This proves ``\reff{le:count rect:loc Lip}$\then$\reff{le:count rect:def}'' and completes the proof of Lemma \ref{le:count rect}.
\end{proof}

The following lemma was used in Remark \ref{rmk:inj sympl}\reff{rmk:inj sympl:emb}. We have learned this lemma and its proof from an answer of harfe to a question on stackexchange, see \cite{harfe}.
\begin{lemma}[injectivity of restriction of locally injective continuous function that is injective on compact set]\label{le:loc inj compact} Let $X,Y$ be topological spaces with $Y$ Hausdorff, $K\sub X$ a compact subset, and $f:X\to Y$ a locally injective continuous map that is injective on $K$. Then there exists a neighbourhood $U$ of $K$ in $X$ on which $f$ is injective.
\end{lemma}
For the convenience of the reader we repeat the proof of this lemma in \cite{harfe} in a slightly rephrased form.
\begin{proof}[Proof of Lemma \ref{le:loc inj compact}]\setcounter{claim}{0} Let $y\in f(K)$.
\begin{claim}\label{claim:U V} There exist an open neighbourhood $U^y$ of $K$ in $X$ and an open neighbourhood $V^y$ of $y$ in $Y$, such that 
\begin{equation}\label{eq:forall x}\forall x,x'\in U^y:\quad x=x'\textrm{ or }f(x)\neq f(x')\textrm{ or }f(x)\not\in V^y.
\end{equation}
\end{claim}
\begin{proof}[Proof of Claim \ref{claim:U V}] For every $x\in K$ there exist open neighbourhoods $U_x$ of $x$ and $V_x$ of $y$, such that $f$ is injective on $U_x$, and
\begin{equation}\label{eq:Ux cap}f(x)\neq y\then U_x\cap f^{-1}(V_x)=\emptyset.\end{equation}
Here we used that $f$ is locally injective and continuous and that $Y$ is Hausdorff. Since $K$ is compact, there exists a finite subset $S\sub K$, such that
\[K\sub U^y:=\bigcup_{x_0\in S}U_{x_0}.\]
This set is open and hence an open neighbourhood of $K$. We define
\[V^y:=\bigcap_{x_0\in S}V_{x_0}.\]
This is an open neighbourhood of $y$. We check \eqref{eq:forall x}. Let $x,x'\in U^y$, such that
\begin{equation}\label{eq:f x}f(x)=f(x')\in V^y.\end{equation}
We claim that
\begin{equation}\label{eq:x=x'}x=x'.\end{equation}
To see this, we choose $x_0,x_0'\in S$, such that $x\in U_{x_0}$, $x'\in U_{x_0'}$. By \eqref{eq:f x} we have $f(x)\in V^y\sub V_{x_0}$ and hence $x\in U_{x_0}\cap f^{-1}(V_{x_0})$. Hence this set is nonempty. Therefore, by \eqref{eq:Ux cap}, we have $f(x_0)=y$. Similarly, we have $f(x_0')=y=f(x_0)$. Since $x_0,x_0'\in S\sub K$ and $f$ is injective on $K$, it follows that $x_0=x_0'$.

Using $x\in U_{x_0}$, $x'\in U_{x_0'=x_0}$, \eqref{eq:f x}, and injectivity of $f$ on $U_{x_0}$, it follows that $x=x'$. This proves \eqref{eq:x=x'}. Hence \eqref{eq:forall x} holds. Hence $U^y,V^y$ have the desired properties. This proves Claim \ref{claim:U V}.
\end{proof}
We choose $U^y,V^y$ as in Claim \ref{claim:U V}. Since $f(K)$ is compact, there exists a finite subset $S\sub f(K)$, such that
\begin{equation}\label{eq:f K}f(K)\sub\bigcup_{y_0\in S}V^{y_0}.\end{equation}
We define
\begin{equation}\label{eq:U}U:=\bigcap_{y_0\in S}U^{y_0}.\end{equation}
The conclusion of Lemma \ref{le:loc inj compact} follows from the next claim.
\begin{claim}\label{claim:U} The set $U$ has the desired properties.
\end{claim}
\begin{pf}[Proof of Claim \ref{claim:U}] The set $U$ is open and contains $K$. We check that the restriction $f|U$ is injective. Let $x,x'\in U$ be such that $y:=f(x)=f(x')$. By \eqref{eq:f K} there exists $y_0\in S$, such that $y\in V^{y_0}$. By \eqref{eq:U}, we have $x,x'\in U^{y_0}$. Since $f(x)=f(x')=y\in V^{y_0}$, \eqref{eq:forall x} implies that $x=x'$. Hence $f|U$ is injective. This proves Claim \ref{claim:U} \end{pf} and completes the proof of Lemma \ref{le:loc inj compact}.
\end{proof}
		
\bibliographystyle{amsalpha}
\bibliography{biblio}
\end{document}